\documentclass{amsproc}%
\usepackage{amssymb}
\usepackage{graphicx}
\usepackage{amscd}%
\usepackage{amsmath}%
\usepackage{amsfonts}
\theoremstyle{plain}

\newtheorem{definition}{Definition}
\newtheorem{example}{Example}

\newtheorem{lemma}{Lemma}

\newtheorem{proposition}{Proposition}

\newtheorem{theorem}{Theorem}
\numberwithin{equation}{section}

\begin{document}
\title{A bracket polynomial for graphs}
\author{L. Traldi}
\address{Lafayette College\\
Easton, Pennsylvania 18042}
\email{traldil@lafayette.edu}
\urladdr{http://www.lafayette.edu/\symbol{126}traldil/}
\author{L. Zulli}
\address{Lafayette College\\
Easton, Pennsylvania 18042}
\email{zullil@lafayette.edu}
\urladdr{http://www.cs.lafayette.edu/\symbol{126}zullil/}
\thanks{This paper is in final form and no version of it will be submitted for
publication elsewhere.}
\date{}

\begin{abstract}
A knot diagram has an associated \textit{looped interlacement graph}, obtained
from the intersection graph of the Gauss diagram by attaching loops to the
vertices that correspond to negative crossings. This construction suggests an
extension of the Kauffman bracket to an invariant of looped graphs, and an
extension of Reidemeister equivalence to an equivalence relation on looped
graphs. The graph bracket polynomial can be defined recursively using the same
pivot and local complementation operations used to define the interlace
polynomial, and it gives rise to a graph Jones polynomial $V_{G}(t)$ that is
invariant under the graph Reidemeister moves.

\end{abstract}
\maketitle

\section{Introduction}

Shortly after Jones introduced his polynomial invariant of knots in \cite{J},
Kauffman described the Jones polynomial in the following way \cite{K}. First,
a three-variable bracket polynomial for link diagrams is defined, using either
a state sum or a recursion. Then the variables that appear in the bracket
polynomial are evaluated so that the result is invariant under Reidemeister
moves of the second and third types. Finally, this simplified bracket
polynomial is multiplied by an appropriate factor so that the resulting
product is invariant under all three types of Reidemeister moves.
Thistlethwaite \cite{T} observed that the state sum and recursion used for
Kauffman's bracket imply that there is a strong connection between the bracket
polynomial of a link diagram and the Tutte polynomial of a graph associated
with a checkerboard coloring of the diagram's complementary regions.

\medskip

In this paper we present a fundamentally different graph-theoretic approach to
knot-theoretic ideas related to the Jones polynomial. We take a moment to
outline this approach before beginning a detailed presentation.

\medskip

First we introduce a new graph invariant, the \textit{graph bracket
polynomial}. If $G$ is an $n$-vertex graph then $[G]$ can be described in two
different ways. One description is a sum indexed by the $n\times n$ diagonal
matrices over $GF(2)$; this sum is a direct extension of a formula for the
Jones polynomial of a classical knot given in \cite{Z}. The other description
is a recursion involving the local complementation and pivot operations of
\cite{A1, A2, A, B}. The interlace polynomials of \cite{A1, A2, A} can also be
defined using either sums of $2^{n}$ terms involving $n\times n$ matrices over
$GF(2)$ or recursions involving the local complementation and pivoting
operations, but the graph bracket and interlace polynomials seem to be
genuinely different both in definition and in significance. For instance,
Sections 5 and 6 of \cite{A} suggest that the interlace polynomials capture
more information about the independent sets of a graph than the graph bracket
does; on the other hand the graph bracket captures more knot-theoretic
information than the interlace polynomials do (see \cite{BBRS}).

\medskip

Just as the Kauffman bracket $[D]$ of a knot diagram is related to the Tutte
polynomial through the checkerboard graph, $[D]$ is related to the graph
bracket polynomial through a different construction, a looped version of the
intersection graph of the Gauss diagram. The concise term \textit{Gauss graph}
is already in use, so we\ refer to this graph as the \textit{looped
interlacement graph}, $\mathcal{L}(D)$.

\begin{definition}
If $D$ is a plane diagram of a classical (or virtual) knot then the
\emph{looped interlacement graph} $\mathcal{L}(D)$ has a vertex for each
(classical) crossing in $D$. Two distinct vertices are adjacent in
$\mathcal{L}(D)$ if and only if the corresponding crossings are
\emph{interlaced} in $D $, i.e., while tracing $D$ (in either direction) one
encounters first one crossing, then the other, then the first again, and
finally the second again. $\mathcal{L}(D)$ has a loop at each vertex
corresponding to a negative crossing of $D$.
\end{definition}%

%TCIMACRO{\FRAME{fhFU}{1.6414in}{0.7619in}{0pt}{\Qcb{positive and negative
%crossings}}{}{crossings.eps}{\special{ language "Scientific Word";
%type "GRAPHIC";  maintain-aspect-ratio TRUE;  display "USEDEF";
%valid_file "F";  width 1.6414in;  height 0.7619in;  depth 0pt;
%original-width 2.303in;  original-height 1.0482in;  cropleft "0";
%filename 'crossings.eps';file-properties "XNPEU";}}}%
%BeginExpansion
\begin{figure}
[h]
\begin{center}
\includegraphics[
height=0.7619in,
width=1.6414in
]%
{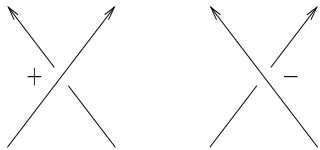}%
\caption{positive and negative crossings}%
\end{center}
\end{figure}
%EndExpansion

$\mathcal{L}(D)$ appeared implicitly in early research regarding the Jones
polynomial. Lannes \cite{L} showed that the Arf invariant of a classical knot
can be obtained from a combinatorial structure that is essentially a loopless
version of $\mathcal{L}(D)$; the relationship between the Arf invariant and
the evaluation $V_{K}(i)$ was deduced from this result in \cite{J}. Zulli
\cite{Z} showed that Kauffman's state sum formula for the Jones polynomial of
a classical knot can be obtained from the adjacency matrix of $\mathcal{L}(D)$
(the \textit{trip matrix} of $D$), but did not consider the graph
$\mathcal{L}(D)$; the Jones polynomial of a virtual knot can be obtained from
$\mathcal{L}(D)$ in the same way.

\medskip

In Section 6 we observe that through the looped interlacement graph
construction, the Reidemeister moves correspond to a family of simple
graph-theoretic operations. Consideration of small examples leads one to guess
that a graph Jones polynomial invariant under the graph Reidemeister moves can
be obtained from the graph bracket, just as for knots; this guess is verified
using the graph bracket's recursive description. As $\mathcal{L}(D)$ is
defined for virtual knot diagrams, these notions share with virtual knot
theory some striking differences from the classical case. A fundamental
difference is the fact that loop-attachment and -removal (the graph operations
that correspond to crossing switches) do not suffice to reduce an arbitrary
graph to a Reidemeister equivalent of an edgeless graph; consequently the
graph Jones polynomial cannot be computed simply by changing loops and
performing Reidemeister moves in order to produce \textquotedblleft
ungraphs.\textquotedblright

\medskip

We close this introduction with the observation that these results suggest
many questions. How is the graph bracket polynomial of a graph $G$ related to
the interlace polynomials of $G$? Are there graph invariants different from
the interlace and graph bracket polynomials that are determined by other
recursions involving the pivot and local complementation operations? Which
invariants of (classical or virtual) knots extend to Reidemeister equivalence
invariants of graphs? Is every graph Reidemeister equivalent to the looped
interlacement graph of some virtual knot diagram? Can these notions be
extended to links of more than one component?

\section{Defining the graph bracket polynomial}

Definition 1 simply specifies which vertex-pairs are to be adjacent in
$\mathcal{L}(D)$, and which vertices are to carry loops; multiple edges do not
occur in looped interlacement graphs, and modifying the definition to allow
multiple edges would not make looped interlacement graphs more useful. We
assume similarly that all the graphs we consider in this paper have no
multiple edges. The reader who is a stickler for generality may prefer to
consider arbitrary (multi-)graphs, with the understanding that every set of
multiple edges is to be grouped together and treated as a unit.

\begin{definition}
Let $G$ be a finite graph with vertex-set $V(G)=\{v_{1},...,v_{n}\}$. The
\emph{Boolean adjacency matrix} $\mathcal{A}(G)=(a_{ij})$ is the $n\times n$
matrix over $GF(2)$ whose $ij$ entry is 1 or 0 according to whether or not $G
$ has an edge $\{v_{i},v_{j}\}$.
\end{definition}

\begin{definition}
The \emph{graph bracket polynomial} of a finite graph $G$ is
\[
\lbrack G](A,B,d)=\sum_{\Delta}A^{\nu(\Delta)}B^{\rho(\Delta)}d^{\nu
(\mathcal{A}(G)+\Delta)}%
\]
with a summand for each $n\times n$ diagonal matrix $\Delta$ over $GF(2)$;
here $\nu$ and $\rho$ denote the nullity and rank of matrices over $GF(2)$, respectively.
\end{definition}

Observe that $[G]$ is an isomorphism invariant, i.e., it is not affected by
the ordering of $V(G)$. Note also that no information in $[G]$ is lost if we
replace $B$ by $A^{-1}$ or $1$;\ we include $B$ in Definition 3 only to agree
with Kauffman's original bracket notation \cite{K}.

\medskip

Definition 3 is motivated by the observation that if $D$ is a diagram of a
classical knot $K$ then $\mathcal{A}(\mathcal{L}(D))$ is the \emph{trip
matrix} of $D$ and hence Theorem 2 of \cite{Z} implies that the Kauffman
bracket $[D]$ equals the graph bracket $[\mathcal{L}(D)]$. It follows that if
$\mathcal{L}(D)$ has $n$ vertices and $\ell$ loops then the Jones polynomial
$V_{K}(t)$ is the image of $(-A^{3})^{\ell}(-B^{3})^{n-\ell}[\mathcal{L}(D)]$
under the evaluations $A\mapsto t^{-1/4}$, $B\mapsto t^{1/4} $ and $d$
$\mapsto-t^{-1/2}-t^{1/2}$. The computational intractability of $V_{K}(t)$
\cite{JVW} implies that calculating the graph bracket polynomial is also
intractable in general.

\medskip

Definition 3 bears some resemblance to the definition of the two-variable
interlace polynomial $q(G)$ \cite{A}:
\[
q(G)=\sum_{S\subseteq V(G)}(x-1)^{\rho(\mathcal{A}(G[S]))}(y-1)^{\nu
(\mathcal{A}(G[S]))}%
\]

\noindent where $G[S]$ is the subgraph of $G$ induced by $S$. Note that $q(G)
$ can be obtained from the sum
\[
\sum_{S\subseteq V(G)}[G[S]]
\]

\noindent by evaluating $A\mapsto x-1$, $B\mapsto0$, and $d\mapsto
(y-1)/(x-1)$; the evaluation $B\mapsto0$ eliminates most of the terms in the
sum, so this is certainly not an efficient way to obtain $q(G)$. There is an
even more inefficient way to obtain $[G]$ from a sum involving interlace
polynomials. For an $n\times n$ diagonal matrix $\Delta$ over $GF(2)$ let
$G+\Delta$ denote the graph obtained from $G$ by toggling the loops at the
vertices corresponding to nonzero entries of $\Delta$, so that $G+\Delta$ has
a loop at a vertex $v_{i}$ if and only if either $G$ has a loop at $v_{i}$ and
the $i^{th}$ entry of $\Delta$ is 0, or $G$ does not have a loop at $v_{i}$
and the $i^{th}$ entry of $\Delta$ is 1. $[G]$ can then be obtained from
\[
\sum_{\Delta}z^{n}A^{\nu(\Delta)}B^{\rho(\Delta)}q(G+\Delta)
\]

\noindent by first evaluating $x\mapsto z^{-1}+1$, $y\mapsto dz^{-1}+1$ and
then evaluating $z\mapsto0$.

\medskip

In addition to Definition 3, the graph bracket has a recursive definition
which involves the local complementation and pivoting operations used by
Bouchet \cite{B} and Arratia, Bollob\'{a}s and Sorkin \cite{A1, A2, A}.

\begin{definition}
(Local Complementation) Let $G$ be a finite graph. If $a$ is a vertex of $G$
then $G^{a}$ is obtained from $G$ by toggling\ adjacencies $\{x,y\}$ involving
neighbors of $a$ that are distinct from $a$. That is, if $x\neq a\neq y$ and
$x,y$ are neighbors of $a$ in $G$ then $G^{a}$ contains an edge $\{x,y\}$ if
and only if $G$ does not.
\end{definition}

Note that the definition allows for the possibility that $x=y$, in which case
the edge $\{x,y\}$ is a loop. We will not usually refer to the entire graph
$G^{a}$, but rather the subgraph $G^{a}-a$ obtained from $G^{a}$ by removing
$a$ and all edges incident on $a$.

\begin{definition}
(Pivot) Let $G$ be a finite graph with distinct vertices $a$ and $b$. Then the
graph $G^{ab}$ is obtained from $G$ by toggling\ adjacencies $\{x,y\}$ such
that $x,y\notin\{a,b\}$, $x$ is adjacent to $a$ in $G$, $y$ is adjacent to $b$
in $G$, and either $x$ is not adjacent to $b$ or $y$ is not adjacent to $a$.
That is, $G^{ab}$ contains such an edge $\{x,y\}$ if and only if $G$ does not.
\end{definition}

The recursive definition of the graph bracket is given in Theorem 1. As we
will see in\ Section 4 below, part (i) is an extension to the graph bracket of
the switching\ formula of the Kauffman bracket \cite{Kd} and the braid-plat
formula of the Jones polynomial \cite{BK}.

\begin{theorem}
(i) If $G$ is a finite graph with a loop at $a$ then
\[
\lbrack G]=A^{-1}B[G-\{a,a\}]+(A-A^{-1}B^{2})[G^{a}-a],
\]
where $G-\{a,a\}$ is obtained from $G$ by removing the loop at $a$.

\medskip

(ii) If $a$ and $b$ are distinct loopless neighbors in $G$ then
\[
\lbrack G]=A^{2}[G^{ab}-a-b]+AB[(G^{ab})^{a}-a-b]+B[G^{a}-a].
\]

(iii) The empty graph $E_{0}$ has $[E_{0}]=1$, and the edgeless graph $E_{n}$
with $n\geq1$ vertices has $[E_{n}]=(Ad+B)^{n}$.
\end{theorem}

If $G$ is an $n$-vertex graph then a recursive calculation of $[G]$ using
Theorem 1 yields a formula $[G]=\sum_{i=0}^{n}c_{i}[E_{i}]$ for some
coefficients $c_{0},...,c_{n}$ which are integer polynomials in $A$, $A^{-1}$
and $B$. These coefficients are\ uniquely determined: $c_{n}$ is the
coefficient of $d^{n}$ in $A^{-n}[G]$, $c_{n-1}$ is the coefficient of
$d^{n-1}$ in $A^{1-n}[G]-A^{1-n}c_{n}(Ad+B)^{n}$, and so on. It follows that
the graph bracket polynomial is \textquotedblleft universal\textquotedblright%
\ among graph invariants that satisfy parts (i) and (ii) of Theorem 1, i.e.,
any such graph invariant is determined by the graph bracket. For instance,
suppose $A,B,X_{0},X_{1},...$ are independent indeterminates, with $A$
invertible; then there is a graph invariant $[[G]]$ that satisfies parts (i)
and (ii) of Theorem 1 and has $[[E_{n}]]=X_{n}$ for every $n$. One might
expect this invariant to be more sensitive than the graph bracket, because it
involves infinitely many indeterminates. Instead it contains precisely the
same information as the graph bracket does: for every graph $G$, $[[G]]=$
$\sum_{i=0}^{\left\vert V(G)\right\vert }c_{i}X_{i} $ and $[G]=\sum
_{i=0}^{\left\vert V(G)\right\vert }c_{i}[E_{i}]$, with the same coefficients
$c_{i}$.

\medskip

We proceed to prove Theorem 1. Observe first that $[G]$ is unchanged if we
permute the vertices of $G$, so we may always presume that the vertices of $G
$ are ordered in a convenient way.

\medskip

Suppose $G$ has a loop at $a=v_{1}$, and let%

\[
\sum_{0}=\sum_{\Delta=(\delta_{ij})~with~\delta_{11}=0}A^{\nu(\Delta)}%
B^{\rho(\Delta)}d^{\nu(\mathcal{A}(G)+\Delta)}%
\]

\medskip%

\[
\mathrm{and\quad}\sum_{1}=\sum_{\Delta=(\delta_{ij})~with~\delta_{11}=1}%
A^{\nu(\Delta)}B^{\rho(\Delta)}d^{\nu(\mathcal{A}(G)+\Delta)}.
\]

\medskip

\noindent Then $[G]=\sum_{0}+\sum_{1}$ and $[G-\{a,a\}]=A^{-1}B\sum
_{0}+AB^{-1}\sum_{1}$, so
\[
\lbrack G]=A^{-1}B[G-\{a,a\}]+(1-A^{-2}B^{2})\sum_{0}.
\]

Suppose $\Delta=(\delta_{ij})$ is a diagonal matrix with $\delta_{11}=0$.
Then
\[
\mathcal{A}(G)+\Delta=
\begin{pmatrix}
1 & \mathbf{1} & \mathbf{0}\\
\mathbf{1} & M_{11} & M_{12}\\
\mathbf{0} & M_{21} & M_{22}%
\end{pmatrix}
,
\]

\noindent where bold characters indicate row and column vectors with all
entries the same. If $M_{11}^{c}$ denotes the matrix obtained by changing
every entry in $M_{11}$ then
\[
\nu(\mathcal{A}(G)+\Delta)=\nu%
\begin{pmatrix}
M_{11}^{c} & M_{12}\\
M_{21} & M_{22}%
\end{pmatrix}
=\nu(\mathcal{A}(G^{a}-a)+\Delta^{\prime}),
\]
where $\Delta^{\prime}$ is the submatrix of $\Delta$ obtained by removing the
first row and column. Consequently
\begin{align*}
\sum_{0}  & =\sum_{n\times n~\Delta~with~\delta_{11}=0}A^{\nu(\Delta)}%
B^{\rho(\Delta)}d^{\nu(\mathcal{A}(G)+\Delta)}\\
& =\sum_{(n-1)\times(n-1)~\Delta^{\prime}}A^{1+\nu(\Delta^{\prime})}%
B^{\rho(\Delta^{\prime})}d^{\nu(\mathcal{A}(G^{a}-a)+\Delta^{\prime})}%
=A[G^{a}-a].\hspace{0.7in}%
\end{align*}

This completes the proof of part (i) of Theorem 1. Suppose $a=v_{1}$ and
$b=v_{2}$ are adjacent in $G$, and $G$ does not have a loop at $a$ or $b$. For
$\beta\in\{0,1\}$ let
\[
\sum_{0,\beta}=\sum_{\substack{\Delta=(\delta_{ij})~with \\\delta
_{11}=0\,and\,\delta_{22}=\beta}}A^{\nu(\Delta)}B^{\rho(\Delta)}%
d^{\nu(\mathcal{A}(G)+\Delta)}.
\]

Let $\Delta$ be a diagonal matrix with $\delta_{11}=0=\delta_{22}$. Then
\[
\mathcal{A}(G)+\Delta=\left(
\begin{array}
[c]{cccccc}%
0 & 1 & \mathbf{1} & \mathbf{1} & \mathbf{0} & \mathbf{0}\\
1 & 0 & \mathbf{1} & \mathbf{0} & \mathbf{1} & \mathbf{0}\\
\mathbf{1} & \mathbf{1} & M_{11} & M_{12} & M_{13} & M_{14}\\
\mathbf{1} & \mathbf{0} & M_{21} & M_{22} & M_{23} & M_{24}\\
\mathbf{0} & \mathbf{1} & M_{31} & M_{32} & M_{33} & M_{34}\\
\mathbf{0} & \mathbf{0} & M_{41} & M_{42} & M_{43} & M_{44}%
\end{array}
\right)  .
\]
Using row and column operations to eliminate the $\mathbf{1}$ vectors, we see
that the nullity of $\mathcal{A}(G)+\Delta$ is the same as that of
\[
\left(
\begin{array}
[c]{cccccc}%
0 & 1 & \mathbf{0} & \mathbf{0} & \mathbf{0} & \mathbf{0}\\
1 & 0 & \mathbf{1} & \mathbf{0} & \mathbf{1} & \mathbf{0}\\
\mathbf{1} & \mathbf{0} & M_{11}^{c} & M_{12}^{c} & M_{13} & M_{14}\\
\mathbf{1} & \mathbf{0} & M_{21} & M_{22} & M_{23} & M_{24}\\
\mathbf{0} & \mathbf{0} & M_{31}^{c} & M_{32}^{c} & M_{33} & M_{34}\\
\mathbf{0} & \mathbf{0} & M_{41} & M_{42} & M_{43} & M_{44}%
\end{array}
\right)  \,\mathrm{and}\,\left(
\begin{array}
[c]{cccccc}%
0 & 1 & \mathbf{0} & \mathbf{0} & \mathbf{0} & \mathbf{0}\\
1 & 0 & \mathbf{0} & \mathbf{0} & \mathbf{0} & \mathbf{0}\\
\mathbf{0} & \mathbf{0} & M_{11} & M_{12}^{c} & M_{13}^{c} & M_{14}\\
\mathbf{0} & \mathbf{0} & M_{21}^{c} & M_{22} & M_{23}^{c} & M_{24}\\
\mathbf{0} & \mathbf{0} & M_{31}^{c} & M_{32}^{c} & M_{33} & M_{34}\\
\mathbf{0} & \mathbf{0} & M_{41} & M_{42} & M_{43} & M_{44}%
\end{array}
\right)  .
\]
Observe that
\[
\left(
\begin{array}
[c]{cccc}%
M_{11} & M_{12}^{c} & M_{13}^{c} & M_{14}\\
M_{21}^{c} & M_{22} & M_{23}^{c} & M_{24}\\
M_{31}^{c} & M_{32}^{c} & M_{33} & M_{34}\\
M_{41} & M_{42} & M_{43} & M_{44}%
\end{array}
\right)  =\mathcal{A}(G^{ab}-a-b)+\Delta^{\prime},
\]
where $\Delta^{\prime}$ is obtained from $\Delta$ by removing the first two
rows and columns. As $\rho(\Delta)=\rho(\Delta^{\prime})$ and $\nu(\Delta
)=\nu(\Delta^{\prime})+2$, summing over all such $\Delta$ tells us that
\[
\sum_{0,0}=A^{2}[G^{ab}-a-b].
\]

If $\Delta$ is a diagonal matrix with $\delta_{11}=0\neq\delta_{22}$ then the
nullity of
\[
\mathcal{A}(G)+\Delta=\left(
\begin{array}
[c]{cccccc}%
0 & 1 & \mathbf{1} & \mathbf{1} & \mathbf{0} & \mathbf{0}\\
1 & 1 & \mathbf{1} & \mathbf{0} & \mathbf{1} & \mathbf{0}\\
\mathbf{1} & \mathbf{1} & M_{11} & M_{12} & M_{13} & M_{14}\\
\mathbf{1} & \mathbf{0} & M_{21} & M_{22} & M_{23} & M_{24}\\
\mathbf{0} & \mathbf{1} & M_{31} & M_{32} & M_{33} & M_{34}\\
\mathbf{0} & \mathbf{0} & M_{41} & M_{42} & M_{43} & M_{44}%
\end{array}
\right)
\]
is the same as the nullity of
\[
\left(
\begin{array}
[c]{cccccc}%
0 & 1 & \mathbf{0} & \mathbf{0} & \mathbf{0} & \mathbf{0}\\
1 & 0 & \mathbf{0} & \mathbf{1} & \mathbf{1} & \mathbf{0}\\
\mathbf{1} & \mathbf{0} & M_{11}^{c} & M_{12}^{c} & M_{13} & M_{14}\\
\mathbf{1} & \mathbf{0} & M_{21} & M_{22} & M_{23} & M_{24}\\
\mathbf{0} & \mathbf{0} & M_{31}^{c} & M_{32}^{c} & M_{33} & M_{34}\\
\mathbf{0} & \mathbf{0} & M_{41} & M_{42} & M_{43} & M_{44}%
\end{array}
\right)  \,\mathrm{and}\,\left(
\begin{array}
[c]{cccccc}%
0 & 1 & \mathbf{0} & \mathbf{0} & \mathbf{0} & \mathbf{0}\\
1 & 0 & \mathbf{0} & \mathbf{0} & \mathbf{0} & \mathbf{0}\\
\mathbf{0} & \mathbf{0} & M_{11}^{c} & M_{12} & M_{13}^{c} & M_{14}\\
\mathbf{0} & \mathbf{0} & M_{21} & M_{22}^{c} & M_{23}^{c} & M_{24}\\
\mathbf{0} & \mathbf{0} & M_{31}^{c} & M_{32}^{c} & M_{33} & M_{34}\\
\mathbf{0} & \mathbf{0} & M_{41} & M_{42} & M_{43} & M_{44}%
\end{array}
\right)  .
\]
If $\Delta^{\prime}$ is obtained from $\Delta$ by removing the first two rows
and columns then
\[
\left(
\begin{array}
[c]{cccc}%
M_{11}^{c} & M_{12} & M_{13}^{c} & M_{14}\\
M_{21} & M_{22}^{c} & M_{23}^{c} & M_{24}\\
M_{31}^{c} & M_{32}^{c} & M_{33} & M_{34}\\
M_{41} & M_{42} & M_{43} & M_{44}%
\end{array}
\right)  =\mathcal{A}((G^{ab})^{a}-a-b)+\Delta^{\prime}.
\]
As $\rho(\Delta)=\rho(\Delta^{\prime})+1$ and $\nu(\Delta)=\nu(\Delta^{\prime
})+1$, summing over all such $\Delta$ shows that
\[
\sum_{0,1}=AB[(G^{ab})^{a}-a-b].
\]

Let
\[
\sum_{1}=\sum_{\Delta=(\delta_{ij})~with~\delta_{11}=1}A^{\nu(\Delta)}%
B^{\rho(\Delta)}d^{\nu(\mathcal{A}(G)+\Delta)}.
\]
If $\Delta$ is a diagonal matrix with $\delta_{11}=1$ then the nullity of
\[
\mathcal{A}(G)+\Delta=
\begin{pmatrix}
1 & \mathbf{1} & \mathbf{0}\\
\mathbf{1} & M_{11} & M_{12}\\
\mathbf{0} & M_{21} & M_{22}%
\end{pmatrix}
\]
is the same as the nullity of
\[%
\begin{pmatrix}
M_{11}^{c} & M_{12}\\
M_{21} & M_{22}%
\end{pmatrix}
=\mathcal{A}(G^{a}-a)+\Delta^{\prime},
\]
where $\Delta^{\prime}$ is obtained from $\Delta$ by removing the first row
and column. As $\rho(\Delta)=\rho(\Delta^{\prime})+1$ and $\nu(\Delta
)=\nu(\Delta^{\prime})$, summing over all such $\Delta$ tells us that
\[
\sum_{1}=B[G^{a}-a].
\]

The above formulas for $\sum_{0,0}$, $\sum_{0,1}$ and $\sum_{1}$ imply part
(ii) of Theorem 1.

$\medskip$

Part (iii) follows immediately from Definition 3.$\qquad\square$

\section{Some properties of the graph bracket}

\begin{proposition}
If $G$ is the union of disjoint subgraphs $G_{1}$ and $G_{2}$ then
$[G]=[G_{1}]\cdot\lbrack G_{2}].$
\end{proposition}

\begin{proof}
Each matrix $\mathcal{A}(G)+\Delta$ which appears in Definition 3 consists of
two diagonal blocks $\mathcal{A}(G_{1})+\Delta_{1}$ and $\mathcal{A}%
(G_{2})+\Delta_{2}.$
\end{proof}

\begin{proposition}
Let $G$ be a graph and let $G+I$ denote the graph obtained by toggling the
loops at the vertices of $G$ (i.e., $G+I$ has loops at precisely those
vertices where $G$ does not). Then $[G+I](A,B,d)=[G](B,A,d).$
\end{proposition}

\begin{proof}
Let $I$ be the $n\times n$ identity matrix. Then
\begin{align*}
\lbrack G+I](A,B,d)  & =\sum_{\Delta}A^{\nu(\Delta)}B^{\rho(\Delta)}%
d^{\nu(\mathcal{A}(G+I)+\Delta)}\\
& =\sum_{\Delta}A^{\rho(I+\Delta)}B^{\nu(I+\Delta)}d^{\nu(\mathcal{A}%
(G)+I+\Delta)}=[G](B,A,d).
\end{align*}

\noindent This argument is essentially the same as the proof of Proposition 3
of \cite{SZ}.
\end{proof}

Two simplifications of the graph bracket are defined just as for knots.

\begin{definition}
The reduced graph bracket of $G$ is $\left\langle G\right\rangle
(A)=[G](A,A^{-1},-A^{2}-A^{-2})$. If $G$ has $n$ vertices and $\ell$ loops
then the graph Jones polynomial of $G$ is $V_{G}(t)=(-1)^{n}\cdot
t^{(3n-6\ell)/4}\cdot\left\langle G\right\rangle (t^{-1/4})$.
\end{definition}

Some of the properties of Jones polynomials of classical knots that were given
in \cite{J} do not extend to the Jones polynomials of arbitrary graphs. For
instance $V_{K}(t)$ is always a Laurent polynomial in $t$, but\ odd powers of
$t^{1/2}$ may appear in $V_{G}(t)$; e.g., the 2-vertex path $P_{2}$ (a
non-loop edge with its two end-vertices) has $V_{P_{2}}(t)=t^{3/2}%
(t^{-1/2}+1-t)$. (N.b. Definition 3 implies that $\left\langle G\right\rangle
(A)$ involves only powers $A^{k}$ with $k\equiv n~(\operatorname{mod}2)$, so
$V_{G}(t)$ is always a Laurent polynomial in $t^{1/2}$.) It follows that
results about specific values of $V_{K}(t)$ cannot be unambiguously
generalized to $V_{G}(t)$. For example, classical knots have $V_{K}(e^{2\pi
i/3})=1=\pm V_{K}(i)$, but $V_{P_{2}}(e^{2\pi i/3})$ is $-2+i\sqrt{3}$ or $1$
according to the choice of $(e^{2\pi i/3})^{1/2}$, and similarly $V_{P_{2}%
}(i)$ is $i(1\pm\sqrt{2})$. Every classical knot diagram can be changed into a
diagram of the trivial knot by reversing some crossings, and consequently if
$\mathcal{L}(D)$ is the looped interlacement graph of a classical knot diagram
then there is a graph $G$ that differs from $\mathcal{L}(D)$ only with regard
to loops and has $V_{G}(t)=1$. This is another property that does not extend
to the Jones polynomials of arbitrary graphs; for instance, it is not possible
to obtain a graph $G$ with $V_{G}(t)=1$ by adjoining loops to the 6-vertex
path $P_{6}$.

\medskip

By the way, both $P_{2}$ and $P_{6}$ are looped interlacement graphs of
virtual knots, so the \textquotedblleft misbehavior\textquotedblright%
\ mentioned in the preceding paragraph occurs when the Jones polynomial is
extended from classical knots to virtual knots. We do not know if there is a
graph whose Jones polynomial is not the Jones polynomial of some virtual knot.

\medskip

Theorems 15 and 16 of \cite{J}, on the other hand, do extend to arbitrary graphs.

\begin{proposition}
If we use $1$ for $1^{-1/4}$ then every graph has $V_{G}(1)=1$ and
$V_{G}^{\prime}(1)=0$.
\end{proposition}

\begin{proof}
The equality $V_{G}(1)=1$ is equivalent to $\left\langle G\right\rangle
(1)=(-1)^{n}$. This latter is certainly true for $\left\langle E_{n}%
\right\rangle =(-A^{3})^{n}$, and Theorem 1 directly provides a general
inductive proof.

Note that
\begin{align*}
V_{G}^{\prime}(1)  & =(-1)^{n}\left(  \frac{3n-6\ell}{4}\right)
\cdot\left\langle G\right\rangle (1)+(-1)^{n}\frac{d}{dt}(\left\langle
G\right\rangle (t^{-1/4}))(1)\\
& =\left(  \frac{3n-6\ell}{4}\right)  +(-1)^{n}\frac{d}{dt}(\left\langle
G\right\rangle (t^{-1/4}))(1).
\end{align*}
Consequently $V_{G}^{\prime}(1)=0$ if $\frac{d}{dt}(\left\langle
G\right\rangle (t^{-1/4}))(1)=(-1)^{n}\left(  \frac{6\ell-3n}{4}\right)  $;
the equivalent formula $\frac{d\left\langle G\right\rangle }{dA}%
(1)=(-1)^{n}(3n-6\ell)$ is verified in Proposition 4 below.
\end{proof}

Note that Proposition 3 and Theorem 15 of \cite{J} together imply that the
graph bracket (resp. the Jones polynomial) of a graph cannot equal the
Kauffman bracket (resp. the Jones polynomial) of a link of two or more components.

\begin{proposition}
The graph bracket polynomial of $G$ determines\ both the number $n$ of
vertices in $G$ and the number $\ell$ of loops in $G$:
\[
n=\log_{2}\left(  [G](1,1,1)\right)  \qquad and\qquad\ell=\frac{n}{2}-\left(
\frac{(-1)^{n}}{6}\right)  \cdot\frac{d\left\langle G\right\rangle }{dA}(1).
\]

\end{proposition}

\begin{proof}
Definition 3 immediately implies that $[G](A,B,1)=(A+B)^{n}$, and hence
$[G](1,1,1)=2^{n}$.

The equality $\frac{d\left\langle G\right\rangle }{dA}(1)=(-1)^{n}(3n-6\ell) $
is certainly true for $\left\langle E_{n}\right\rangle =(-1)^{n}A^{3n}$.
Theorem 1 provides a general inductive proof as follows.

If $G$ has a loop at $a$ then
\[
\left\langle G\right\rangle =A^{-2}\left\langle G-\{a,a\}\right\rangle
+(A-A^{-3})\left\langle G^{a}-a\right\rangle .
\]
Using the inductive hypothesis and the equality $\left\langle G\right\rangle
(1)=(-1)^{n}$, we conclude that
\begin{align*}
\frac{d\left\langle G\right\rangle }{dA}(1)  & =-2\left\langle
G-\{a,a\}\right\rangle (1)+(-1)^{n}(3n-6(\ell-1))+4\left\langle G^{a}%
-a\right\rangle (1)\\
& =(-1)^{n}\cdot\left(  -2+3n-6\ell+6-4\right)  .
\end{align*}

Suppose $G$ has no loops and $\{a,b\}$ is an edge of $G$. Then
\[
\left\langle G\right\rangle =A^{2}\left\langle G^{ab}-a-b\right\rangle
+\left\langle (G^{ab})^{a}-a-b\right\rangle +A^{-1}\left\langle G^{a}%
-a\right\rangle .
\]
$G^{ab}-a-b$ is loopless, $G^{a}-a$ has $\deg(a)$ loops, and $(G^{ab}%
)^{a}-a-b$ has $\deg(a)-1$ loops; consequently the inductive hypothesis and
the equality $\left\langle G\right\rangle (1)=(-1)^{n}$ imply
\begin{align*}
\frac{d\left\langle G\right\rangle }{dA}(1)  & =2(-1)^{n-2}+(-1)^{n-2}%
(3n-6)+(-1)^{n-2}\left(  3n-6-6(\deg(a)-1)\right) \\
& -(-1)^{n-1}+(-1)^{n-1}(3n-3-6\deg(a)))\\
& =(-1)^{n}\cdot(2+3n-6+3n-6\deg(a)+1-3n+3+6\deg(a)).
\end{align*}

\end{proof}

\section{Applying Theorem 1 to knots}

Theorem 1 results in recursive algorithms that can be used to calculate the
bracket and Jones polynomials of\ (virtual) knot diagrams. These algorithms
have two distinctive features: links of more than one component never appear,
and the recursions involve simple reduction of the crossing number without any
reference to \textquotedblleft unknotting.\textquotedblright\ As knots and
links are usually treated together in the literature regarding the bracket and
Jones polynomials, it may be that these algorithms have not appeared before.%

%TCIMACRO{\FRAME{fhFU}{3.6461in}{1.0058in}{0pt}{\Qcb{$D_{\infty}=D_{B}$ for a
%positive crossing}}{\Qlb{Figure 2}}{grknfi15.ps}%
%width 3.6461in;  height 1.0058in;  depth 0pt;  original-width 8.246in;
%original-height 10.6969in;  cropleft "0";  croptop "0.8501";  cropright "1";
%cropbottom "0.6417";  filename 'grknfi15.ps';file-properties "XNPEU";}} }%
%BeginExpansion
\begin{figure}
[h]
\begin{center}
\includegraphics[
trim=0.000000in 6.864201in 0.000000in 1.603466in,
height=1.0058in,
width=3.6461in
]%
{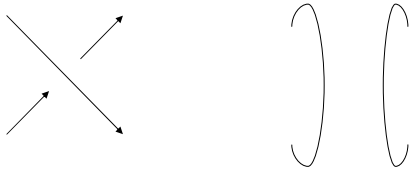}%
\caption{$D_{\infty}=D_{B}$ for a positive crossing}%
\label{Figure 2}%
\end{center}
\end{figure}
%EndExpansion

Clearly crossing-switches in a knot diagram give rise to loop-toggles in the
looped interlacement graph. Moreover, a diagram $D$ of a knot $K$ yields a
diagram $D_{\infty}$ of a knot $K_{\infty}$ by splicing together the two arcs
directed into a crossing, and also splicing together the two arcs directed out
of that crossing, as in Figure 2. The portion of $K_{\infty}$ outside the
pictured area consists of two arcs; the orientation of $K_{\infty}$ along one
arc is the same as that of $K$ while the orientation of $K_{\infty}$ along the
other arc is the reverse of that of $K$. The effect of this partial
orientation-reversal on the looped intersection graph is simple: if $a$ is the
vertex of $\mathcal{L}(D)$ corresponding to the pictured crossing then
$\mathcal{L}(D_{\infty})=\mathcal{L}(D)^{a}-a$. Consequently part (i) of
Theorem 1 is the extension to the graph bracket polynomial of the Jones
polynomial's braid-plat formula ($tV_{-1}-V_{1}=t^{3q}(t-1)V_{\infty}$ in the
notation of \cite{BK}) and the Kauffman bracket's switching formula
($A\chi-A^{-1}\bar{\chi}=(A^{2}-A^{-2})\asymp$ in the notation of \cite{Kd}).

\medskip\medskip

For a positive crossing like the one in Figure 2, $D_{\infty}$ is also denoted
$D_{B}$. The opposite smoothing, pictured in Figure 3, is denoted $D_{0}$ or
$D_{A}$. At a crossing that differs from the pictured one with regard to the
orientation of $K$, the designations of $D_{A}$ and $D_{B}$ are the same but
the designations of $D_{\infty}$ and $D_{0}$ may be reversed; $D_{0}$ always
denotes the diagram obtained by smoothing a crossing of $D$ in a manner
consistent with the orientation of $K$.%

%TCIMACRO{\FRAME{fhFU}{3.6461in}{1.0075in}{0pt}{\Qcb{$D_{0}=D_{A}$ for a
%positive crossing}}{\Qlb{Figure 3}}{grknfi16.ps}%
%width 3.6461in;  height 1.0075in;  depth 0pt;  original-width 8.246in;
%original-height 10.6969in;  cropleft "0";  croptop "0.8503";  cropright "1";
%cropbottom "0.6415";  filename 'grknfi16.ps';file-properties "XNPEU";}} }%
%BeginExpansion
\begin{figure}
[h]
\begin{center}
\includegraphics[
trim=0.000000in 6.862062in 0.000000in 1.601326in,
height=1.0075in,
width=3.6461in
]%
{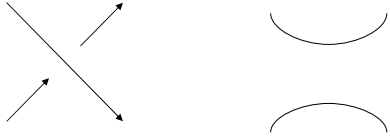}%
\caption{$D_{0}=D_{A}$ for a positive crossing}%
\label{Figure 3}%
\end{center}
\end{figure}
%EndExpansion

Suppose $D$ is an $n$-crossing knot diagram with two interlaced positive
crossings corresponding to vertices $a$ and $b$ of $\mathcal{L}(D)$. We denote
the $(n-2)$-crossing diagrams obtained by smoothing both crossings by first
indicating the smoothing at the $a$ crossing. For instance $D_{\infty0}%
=D_{BB}$ denotes the diagram obtained by smoothing the $a$ crossing against
orientation and then smoothing the $b$ crossing with orientation; as the two
crossings are interlaced the $b$ crossing becomes negative when the $a$
crossing is smoothed, so its orientation-consistent smoothing is $(D_{\infty
})_{B}$ rather than $(D_{\infty})_{A}$. The fundamental recursive formula for
the Kauffman bracket is $[D]=A[D_{A}]+B[D_{B}]$; applying this fundamental
formula twice yields
\[
\lbrack D]=A^{2}[D_{AA}]+AB[D_{AB}]+B[D_{B}].
\]
The extension of this formula to the graph bracket polynomial is part (ii) of
Theorem 1. Consequently, the recursive description of the graph bracket
polynomial given in Theorem 1 specializes to the following algorithm for the
bracket polynomial of a (virtual) knot diagram.

\medskip

1. If $D_{-}$ is a diagram that contains a negative (classical) crossing then
apply the 3-variable form of the switching formula of the Kauffman bracket,
$A[D_{-}]=B[D_{+}]+(A^{2}-B^{2})[D_{\infty}]$, where $D_{+}$ is the diagram
obtained by reversing that negative crossing.

2. If $D$ is a diagram with only positive (classical) crossings, apply the
formula $[D]=A^{2}[D_{AA}]+AB[D_{AB}]+B[D_{B}]$ to any interlaced pair of crossings.

3. Repeat steps 1 and 2 as necessary.

4. A diagram with $n$ crossings, all positive and no two interlaced, has
bracket polynomial $(Ad+B)^{n}$.

\medskip

The resulting algorithm for the Jones polynomial involves formulas that are
slightly more complicated, because obtaining the Jones polynomial from the
bracket involves not only an evaluation but also multiplication by
$(-1)^{n}\cdot t^{3w(D)/4}$. In the notation of \cite{BK}, step 1 involves the
formula $tV_{-1}-V_{1}=t^{3q}(t-1)V_{\infty}$ and step 2 involves the formula
$V_{11}=tV_{00}+t^{3q}V_{0\infty}-t^{1+3q}V_{\infty}$; in each formula $q$ is
the linking number of the two components of the link diagrammed in $D_{0}$.

\section{Examples of graph bracket polynomials}

\begin{example}
Let $K_{n}$ be the complete simple graph on $n$ vertices. Then
\[
\lbrack K_{n}]=\frac{(A+Bd)^{n}-A^{n}}{d}+\left\{
\begin{array}
[c]{c}%
A^{n}d\mathrm{,\,if\,}n\mathrm{\,is\,odd}\\
A^{n}\mathrm{,\,if\,}n\mathrm{\,is\,even}%
\end{array}
\right.  .
\]

\end{example}

\begin{proof}
The proposition is true for $K_{0}=E_{0}$ and $K_{1}=E_{1}$.

Proceeding inductively, suppose $n\geq2$. Proposition 1 and part (ii) of
Theorem 1 tell us that $[K_{n}]=A^{2}[K_{n-2}]+AB[L_{1}]^{n-2}+B[L_{1}]^{n-1}
$, where $L_{1}$ is the graph with one vertex and one loop. Consequently
\begin{align*}
\lbrack K_{n}]  & =A^{2}[K_{n-2}]+AB(A+Bd)^{n-2}+B(A+Bd)^{n-1}\\
& =A^{2}[K_{n-2}]+(A+Bd)^{n-2}\cdot(2AB+B^{2}d)\\
& =\frac{A^{2}(A+Bd)^{n-2}-A^{n}}{d}+(A+Bd)^{n-2}\cdot(2AB+B^{2}d)+\left\{
\begin{array}
[c]{c}%
A^{n}d\mathrm{,\,if\,}n\mathrm{\,is\,odd}\\
A^{n}\mathrm{,\,if\,}n\mathrm{\,is\,even}%
\end{array}
\right. \\
& =\frac{(A+Bd)^{n-2}(A^{2}+2ABd+B^{2}d^{2})-A^{n}}{d}+\left\{
\begin{array}
[c]{c}%
A^{n}d\mathrm{,\,if\,}n\mathrm{\,is\,odd}\\
A^{n}\mathrm{,\,if\,}n\mathrm{\,is\,even}%
\end{array}
\right.  .
\end{align*}

\end{proof}

\begin{example}
\noindent Let $P_{n}$ be the simple path with $n$ vertices, and let $L_{n}$ be
the graph obtained from $P_{n}$ by adding a loop to one vertex at the end of
the path. ($L_{n}$ is a \emph{lollipop}.) Then
\[
\lbrack L_{n}]=(A-A^{-1}B^{2})^{n}+A^{-1}B\sum_{i=0}^{n-1}(A-A^{-1}B^{2}%
)^{i}[P_{n-i}].
\]

\end{example}

\begin{proof}
The proposition holds when $n=1$, for $[L_{1}]=A+Bd=A-A^{-1}B^{2}%
+A^{-1}B(Ad+B)$. Proceeding inductively, if $n>1$ then part (i) of Theorem 1
tells us that $[L_{n}]=A^{-1}B[P_{n}]+(A-A^{-1}B^{2})[L_{n-1}]$.
\end{proof}

\begin{example}
Let $P_{n}$ be the simple path with $n$ vertices. Then
\begin{align*}
\lbrack P_{n}]  & =\gamma_{1}({A^{-1}B^{2}-A)}^{n}+\gamma_{2}(A+B)^{n}\\
& +\gamma_{3}2^{-n}\left(  {-B+\sqrt{4A^{2}-3B^{2}}}\right)  ^{n}+\gamma
_{4}2^{-n}\left(  {-B-\sqrt{4A^{2}-3B^{2}}}\right)  ^{n}%
\end{align*}
\begin{align*}
with~\gamma_{1} &  =0,~\gamma_{2}={\frac{d+2}{3},~\text{$\gamma$}}_{3}%
{=\frac{\left(  4A-3B-\sqrt{4A^{2}-3B^{2}}\right)  (d-1)}{6\sqrt{4A^{2}%
-3B^{2}}}}\\
{~and~{\gamma}_{4}} &  {=~\frac{\left(  3B-4A-\sqrt{4A^{2}-3B^{2}}\right)
(d-1)}{6\sqrt{4A^{2}-3B^{2}}}.}%
\end{align*}

\end{example}

\begin{proof}
Suppose for the moment that $n\geq4$. Example 2\ and part (ii) of Theorem 1
tell us that
\[
\lbrack P_{n}]=(A^{2}+AB)[P_{n-2}]+B[L_{n-1}]
\]
\[
=(A^{2}+AB)[P_{n-2}]+B(A-A^{-1}B^{2})^{n-1}+A^{-1}B^{2}\sum_{i=0}%
^{n-2}(A-A^{-1}B^{2})^{i}[P_{n-1-i}]
\]
Comparing this with
\[
\lbrack P_{n-2}]=(A^{2}+AB)[P_{n-4}]+B(A-A^{-1}B^{2})^{n-3}+A^{-1}B^{2}%
\sum_{i=0}^{n-4}(A-A^{-1}B^{2})^{i}[P_{n-3-i}]
\]
we see that
\begin{align*}
\lbrack P_{n}]  & =(A-A^{-1}B^{2})^{2}[P_{n-2}]-(A-A^{-1}B^{2})^{2}%
(A^{2}+AB)[P_{n-4}]\\
& +(A^{2}+AB)[P_{n-2}]+A^{-1}B^{2}[P_{n-1}]+A^{-1}B^{2}(A-A^{-1}B^{2}%
)[P_{n-2}]
\end{align*}
\[
=A^{-1}B^{2}[P_{n-1}]+\left(  A+B\right)  \left(  2A-B\right)  [P_{n-2}%
]-(A^{2}+AB)(A-A^{-1}B^{2})^{2}[P_{n-4}].
\]

\medskip\noindent Consequently the $[P_{n}]$ satisfy a linear recurrence with
characteristic polynomial
\begin{align*}
& x^{4}-A^{-1}B^{2}x^{3}-\left(  A+B\right)  \left(  2A-B\right)  x^{2}%
+(A^{2}+AB)(A-A^{-1}B^{2})^{2}\\
& =\left(  x+A-A^{-1}B^{2}\right)  (x-A-B)\left(  x^{2}+Bx-A^{2}+B^{2}\right)
;
\end{align*}
this implies that
\begin{align*}
\lbrack P_{n}]  & =\gamma_{1}({A^{-1}B^{2}-A)}^{n}+\gamma_{2}(A+B)^{n}\\
& +\gamma_{3}2^{-n}\left(  {-B+\sqrt{4A^{2}-3B^{2}}}\right)  ^{n}+\gamma
_{4}2^{-n}\left(  {-B-\sqrt{4A^{2}-3B^{2}}}\right)  ^{n}%
\end{align*}
for some coefficients $\gamma_{i}$. The reader may verify that the $\gamma
_{i}$ given in the statement yield the correct values of $[P_{n}]$ for $0\leq
n\leq3$.
\end{proof}

Here is another version of the formula of Example 3.%

\begin{align*}
\lbrack P_{n}]  & =\left(  {\frac{d+2}{3}}\right)  (A+B)^{n}-\left(
{\frac{d-1}{3\cdot2^{n}}}\right)  \left\{
\begin{array}
[c]{r}%
\left(  4A^{2}-3B^{2}\right)  ^{n/2}\mathrm{,\,if\,}n\mathrm{\,is\,even}\\
0\mathrm{,\,if\,}n\mathrm{\,is\,odd}%
\end{array}
\right. \\
& ~\\
& +\left(  {\frac{d-1}{3\cdot2^{n-2}}}\right)  A\sum_{j=0}^{\left\lfloor
\frac{n-1}{2}\right\rfloor }(4A^{2}-3B^{2})^{j}(-B)^{n-2j-1}\binom{n}{2j+1}\\
& ~\\
& +\left(  {\frac{d-1}{3\cdot2^{n}}}\right)  \sum_{j=0}^{\left\lfloor
\frac{n-1}{2}\right\rfloor }(4A^{2}-3B^{2})^{j}(-B)^{n-2j}\left(  3\binom
{n}{2j+1}-\binom{n}{2j}\right)  .
\end{align*}

\medskip

Computer calculations indicate that the graph bracket polynomial is
surprisingly effective in distinguishing small graphs: it distinguishes all
the non-isomorphic graphs with no more than 5 vertices, and also the simple
graphs with 6 vertices. (Indeed, the one-variable simplification $[G](A,1,A)$
distinguishes all these graphs.) There are 5027 different graph bracket
polynomials among the 5096 non-isomorphic 6-vertex looped graphs, and 1028
different graph bracket polynomials among the 1044 non-isomorphic simple
7-vertex graphs.

\medskip%
%TCIMACRO{\FRAME{ftbpFU}{4.2341in}{2.8945in}{0pt}{\Qcb{two 6-vertex graphs with
%the same graph bracket polynomial }}{\Qlb{Figure 4}}{grknfi2.ps}%
%width 4.2341in;  height 2.8945in;  depth 0pt;  original-width 8.246in;
%original-height 10.6969in;  cropleft "0.1225";  croptop "0.9168";
%cropright "0.9074";  cropbottom "0.5046";
%filename 'grknfi2.ps';file-properties "XNPEU";}} }%
%BeginExpansion
\begin{figure}
[ptb]
\begin{center}
\includegraphics[
trim=1.010135in 5.397656in 0.763580in 0.889982in,
height=2.8945in,
width=4.2341in
]%
{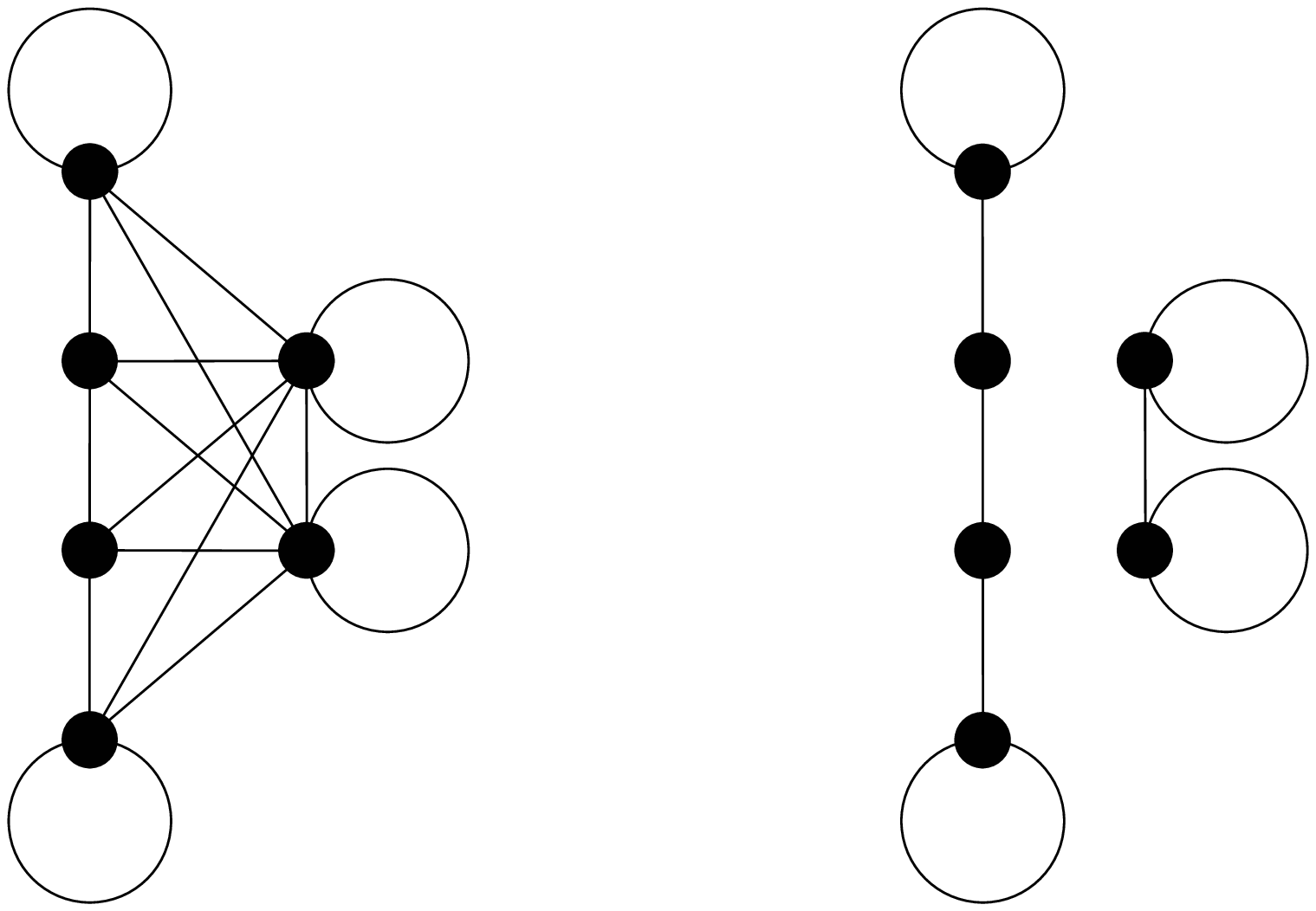}%
\caption{two 6-vertex graphs with the same graph bracket polynomial }%
\label{Figure 4}%
\end{center}
\end{figure}
%EndExpansion

The graph bracket polynomial distinguishes trees with fewer than 7 vertices,
so in general $[G]$ is not determined by the cycle matroid of $G$. Even the
graph Jones polynomial distinguishes some trees, despite being a much less
sensitive invariant than the graph bracket. It would seem, then, that the
graph bracket polynomial provides a novel combinatorial understanding of the
Kauffman bracket and the Jones polynomial, quite different from the very
important relationship between these knot invariants and the Tutte polynomial
originally observed by Thistlethwaite \cite{T}. (See \cite{H, JVW} and their
references for examples of results which have been derived using the
Jones-Tutte relationship.)

\medskip

When the graph bracket polynomial does fail to distinguish two graphs, it can
be surprisingly insensitive to structural differences between them. For
instance, the connected graph on the left in Figure 4 has the same graph
bracket as the disconnected graph on the right.

\section{The knot theory of looped interlacement graphs}

The Reidemeister moves \cite{R} give a combinatorial description of the
relationship among the different diagrams of a given knot type. The following
definition describes the equivalence relation on graphs that is generated by
the \textquotedblleft images\textquotedblright\ of these Reidemeister moves
under the looped interlacement graph construction.%

%TCIMACRO{\FRAME{ftbphFU}{4.3301in}{6.4454in}{0pt}{\Qcb{An $\Omega.3$ move
%involves toggling\ the non-loop edges in one of these six configurations,
%provided that every vertex outside the picture is adjacent to 0 or precisely 2
%of the pictured vertices.}}{\Qlb{Figure 5}}{grjof2.ps}%
%width 4.3301in;  height 6.4454in;  depth 0pt;  original-width 8.246in;
%original-height 10.6969in;  cropleft "0.1304";  croptop "1";  cropright "1";
%cropbottom "0";  filename 'grjof2.ps';file-properties "XNPEU";}} }%
%BeginExpansion
\begin{figure}
[ptbh]
\begin{center}
\includegraphics[
trim=1.075278in 0.000000in 0.000000in 0.000000in,
height=6.4454in,
width=4.3301in
]%
{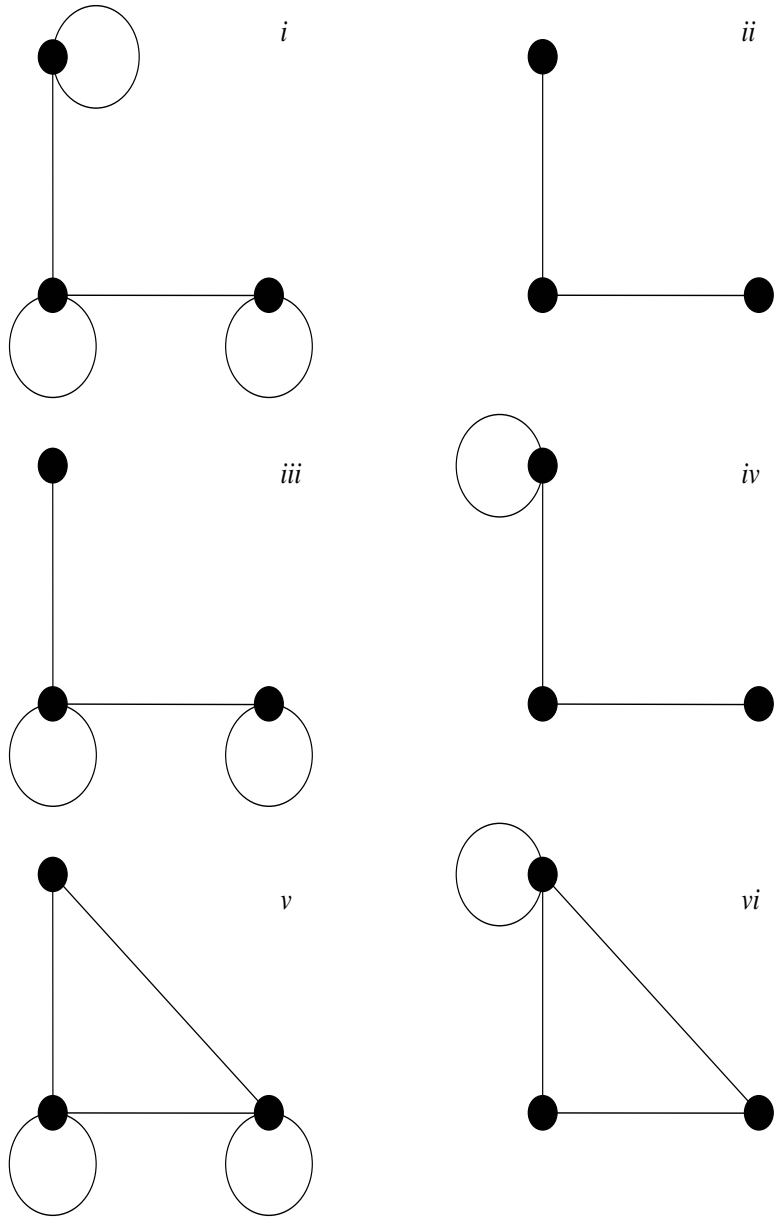}%
\caption{An $\Omega.3$ move involves toggling\ the non-loop edges in one of
these six configurations, provided that every vertex outside the picture is
adjacent to 0 or precisely 2 of the pictured vertices.}%
\label{Figure 5}%
\end{center}
\end{figure}
%EndExpansion

\begin{definition}
Two graphs are \emph{Reidemeister equivalent} if one can be obtained from the
other through some finite sequence of the following \emph{Reidemeister moves}.

$\Omega.1$. Adjoin or remove an isolated vertex. The isolated vertex may be
looped or unlooped.

$\Omega.2$. Adjoin or remove two vertices $v$ and $w$ such that (i) $v$ is
looped, (ii) $w$ is not looped, and (iii) every vertex $x\not \in \{v,w\}$
that is adjacent to one of $v,w$ is also adjacent to the other. The two
vertices $v$ and $w$ may be adjacent or nonadjacent.

$\Omega.3$. Toggle\ the non-loop adjacencies among three distinct vertices
$u$, $v$ and $w$ such that (i) every vertex $x\not \in \{u,v,w\}$ that is
adjacent to any of $u,v,w$ is adjacent to precisely two of $u,v,w$ and (ii)
either the initial or the terminal subgraph spanned by $u,v,w$ is isomorphic
to one of the three-vertex graphs pictured in Figure 5.
\end{definition}

\begin{theorem}
If $D_{1}$ and $D_{2}$ are diagrams of the same (classical or virtual) knot
then the looped interlacement graphs $\mathcal{L}(D_{1})$ and $\mathcal{L}%
(D_{2})$ are Reidemeister equivalent.
\end{theorem}

Let $\mathcal{D}$ denote the class of virtual (possibly classical) knot
diagrams, and let $\mathcal{G}$ denote the class of looped graphs. Theorem 2
implies that the function $\mathcal{L}\colon\mathcal{D}\rightarrow\mathcal{G}$
induces a function ${\widetilde{\mathcal{L}}}\colon\mathcal{D}/\!\sim
~~\rightarrow~\mathcal{G}/\!\sim$, where $\sim$ denotes Reidemeister
equivalence. That is, ${\widetilde{\mathcal{L}}}$ is a knot invariant. This
function is not injective, since the looped interlacement graph of a knot
diagram may be equivalent (even identical) to the looped interlacement graph
of a diagram of a different knot. Not all graphs are equivalent to looped
interlacement graphs of classical knot diagrams, but we do not know whether or
not ${\widetilde{\mathcal{L}}}$ is surjective.

\medskip

Since ${\widetilde{\mathcal{L}}}$ is a knot invariant, so is $f\circ
{\widetilde{\mathcal{L}}}$ for any function $f\colon\mathcal{G}/\!\sim
~~\rightarrow~X$. Such extended invariants are difficult to find because the
graph-theoretic Reidemeister moves affect the structure of a graph
dramatically, changing the number of vertices, number of edges, degree
sequence, connectedness, chromaticity, cycle structure, etc. However the
recursive description of the graph bracket can be used to prove the following.

\begin{theorem}
The graph Jones polynomial $V_{G}$ is invariant under graph-theoretic
Reidemeister equivalence.
\end{theorem}

\section{Six comments on Reidemeister equivalence}

1. Diagrams of distinct classical or virtual knots may have identical looped
interlacement graphs. For instance, this occurs if one diagram is obtained
from another through mutation \cite{CL} or virtualization \cite{Kdv}.

\medskip

2. Some observations of Section 3 imply that several familiar features of
classical knot theory do not extend to Reidemeister equivalence of graphs. For
instance, a graph cannot generally be changed into a Reidemeister equivalent
of an edgeless graph by adjoining or deleting loops.

\medskip%

%TCIMACRO{\FRAME{fhFU}{3.6461in}{3.6184in}{0pt}{\Qcb{{}}}{}{grknif10.ps}%
%width 3.6461in;  height 3.6184in;  depth 0pt;  original-width 8.246in;
%original-height 10.6969in;  cropleft "0";  croptop "0.7860";  cropright "1";
%cropbottom "0.0213";  filename 'grknif10.ps';file-properties "XNPEU";}} }%
%BeginExpansion
\begin{figure}
[h]
\begin{center}
\includegraphics[
trim=0.000000in 0.227844in 0.000000in 2.289137in,
height=3.6184in,
width=3.6461in
]%
{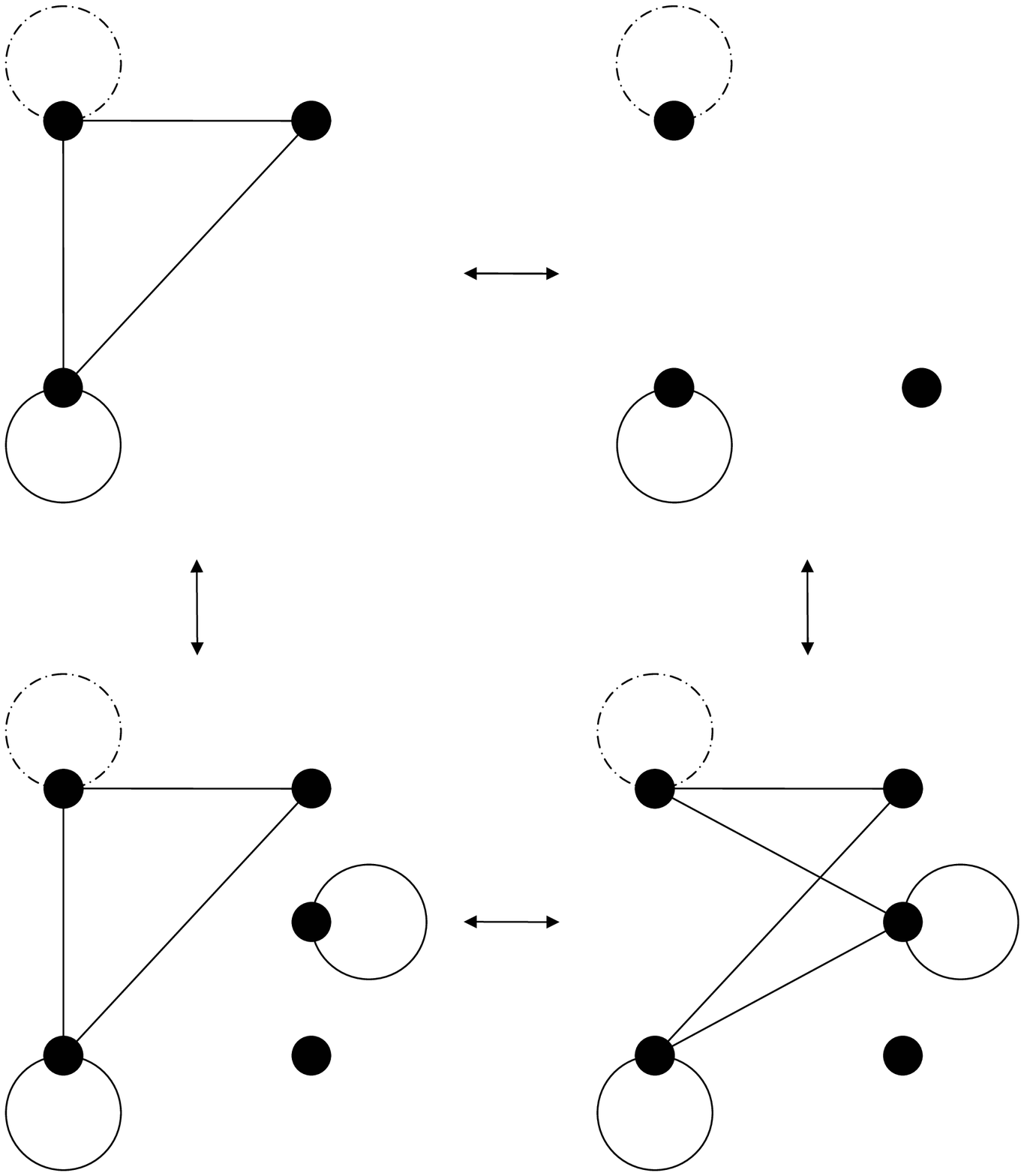}%
\caption{{}}%
\end{center}
\end{figure}
%EndExpansion

\medskip

3. \"{O}stlund \cite{O} showed that the various $\Omega.3$ moves on knot
diagrams are inter-related through composition with $\Omega.2$ moves. As
illustrated in\ Figures 6 and 7, the same observation holds for the graph
Reidemeister moves. In each of these figures, the vertical double-headed
arrows represent $\Omega.2$ moves involving the insertion or deletion of two
vertices whose neighbors outside the picture are the same as those of the
vertex directly above or below them. Without the dashed loop, Figure 6
illustrates the fact that the $\Omega.3$ move pictured in Figure 5 \textit{vi}
can be obtained from the one pictured in Figure 5 \textit{iii} through
composition with $\Omega.2$ moves; with the dashed loop it shows that the
$\Omega.3$ move pictured in Figure 5 \textit{v} can be obtained from the one
pictured in Figure 5 \textit{i}. If all the loops in Figure 6 are toggled the
resulting figure shows that part \textit{vi} of Figure 5 can be obtained from
part \textit{ii} (without the dashed loop) and part \textit{v} from part
\textit{iv} (with the dashed loop). Similarly, without the dashed loop Figure
7 shows that the $\Omega.3$ move pictured in Figure 5 \textit{ii} can be
obtained by composing the one pictured in Figure 5 \textit{vi} with $\Omega.2$
moves; with the dashed loop Figure 7 illustrates that the move pictured in
Figure 5 \textit{iv} can be obtained from the one in Figure 5 \textit{v}.
Toggling the loops in Figure 7 we see that part \textit{iii} of Figure 5 can
be obtained from part \textit{vi} (without the dashed loop) and part
\textit{i} from part \textit{v} (with the dashed loop). In sum, we conclude
that the $\Omega.3$ moves pictured in Figure 5 \textit{i}, \textit{iv} and
\textit{v} can be obtained from each other through composition with $\Omega.2$
moves, and the $\Omega.3$ moves pictured in Figure 5 \textit{ii}, \textit{iii}
and \textit{vi} can also be obtained from each other.

\medskip%

%TCIMACRO{\FRAME{fhFU}{3.6461in}{3.3166in}{0pt}{\Qcb{{}}}{}{grknif11.ps}%
%width 3.6461in;  height 3.3166in;  depth 0pt;  original-width 8.246in;
%original-height 10.6969in;  cropleft "0";  croptop "0.7648";  cropright "1";
%cropbottom "0.0641";  filename 'grknif11.ps';file-properties "XNPEU";}} }%
%BeginExpansion
\begin{figure}
[h]
\begin{center}
\includegraphics[
trim=0.000000in 0.685671in 0.000000in 2.515911in,
height=3.3166in,
width=3.6461in
]%
{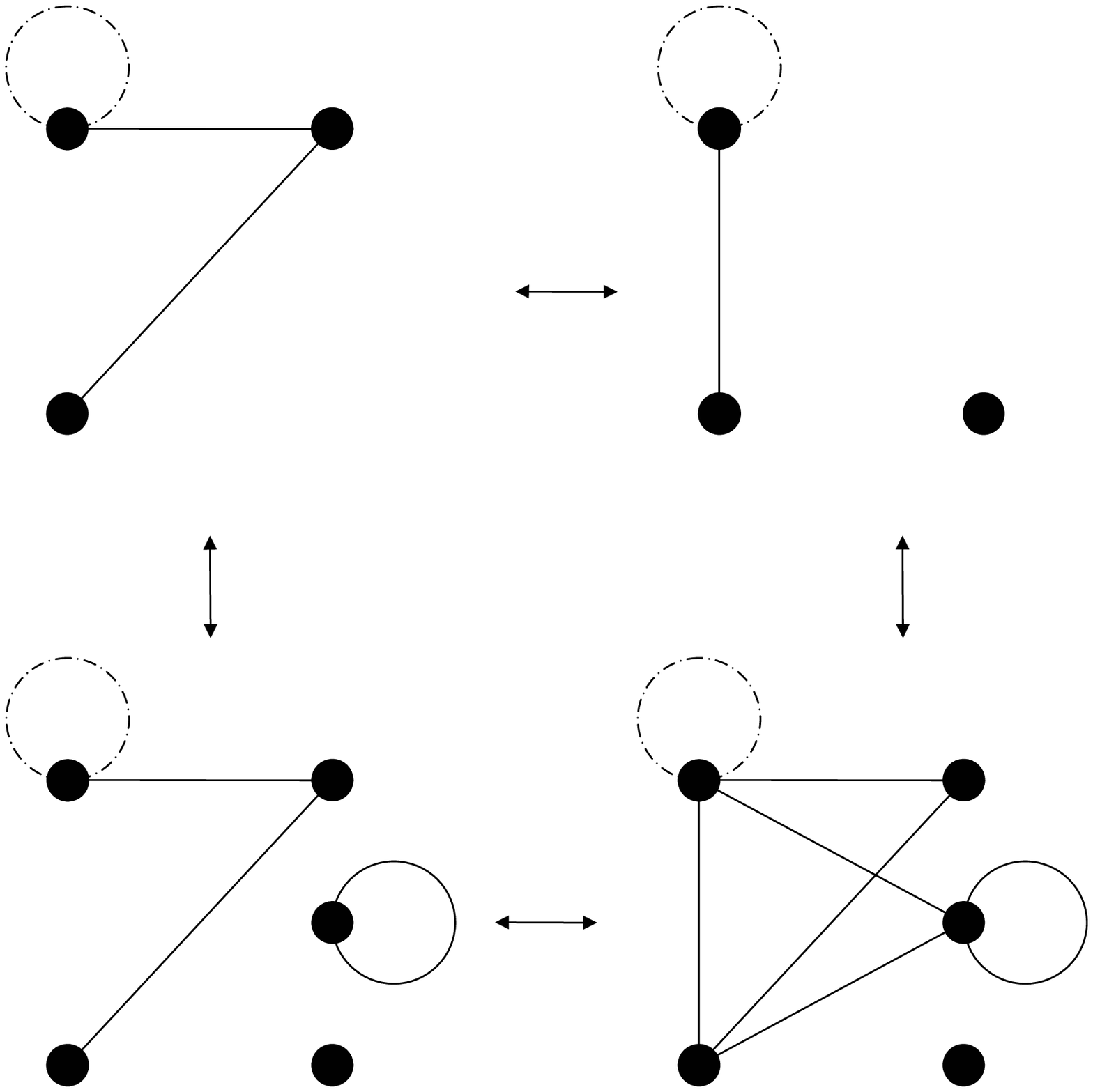}%
\caption{{}}%
\end{center}
\end{figure}
%EndExpansion

4. It will come as no surprise that demonstrating the Reidemeister equivalence
between two graphs can sometimes require the use of intermediate graphs which
are larger than both of the original ones. A simple example appears in Figure
8. Neither of the two graphs that appear at the top of the figure can be
subjected to any Reidemeister move that does not increase the number of
vertices, but as shown in the figure, the two graphs are indeed equivalent.%

%TCIMACRO{\FRAME{fhFU}{4.1511in}{5.3765in}{0pt}{\Qcb{Demonstrating the
%Reidemeister equivalence between the two topmost graphs requires larger
%intermediate graphs.}}{\Qlb{Figure 8}}{grjof3.ps}%
%width 4.1511in;  height 5.3765in;  depth 0pt;  original-width 8.246in;
%original-height 10.6969in;  cropleft "0";  croptop "1";  cropright "1";
%cropbottom "0";  filename 'grjof3.ps';file-properties "XNPEU";}} }%
%BeginExpansion
\begin{figure}
[h]
\begin{center}
\includegraphics[
height=5.3765in,
width=4.1511in
]%
{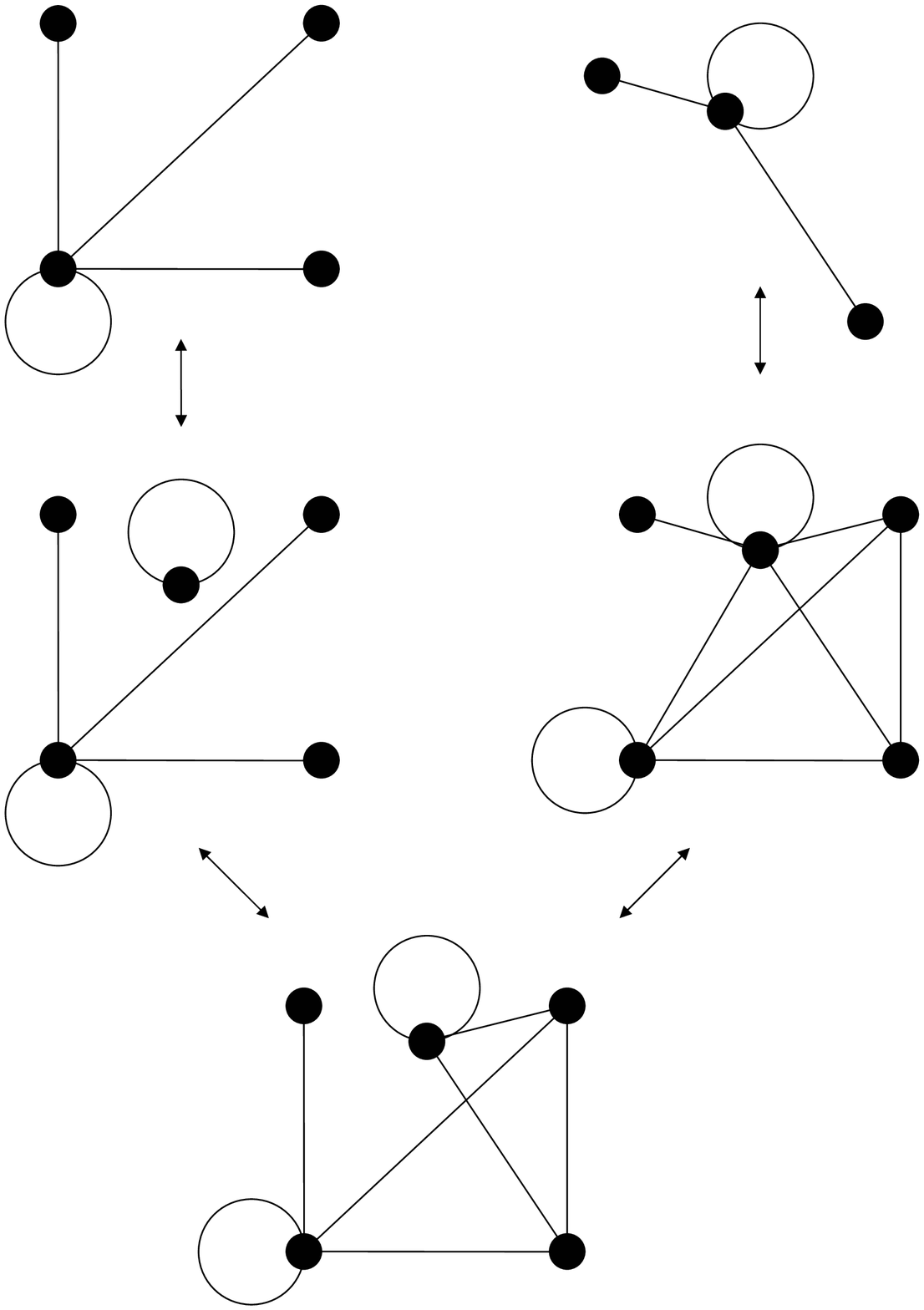}%
\caption{Demonstrating the Reidemeister equivalence between the two topmost
graphs requires larger intermediate graphs.}%
\label{Figure 8}%
\end{center}
\end{figure}
%EndExpansion

\bigskip

5. Gauss \cite{G} observed that the double occurrence words (Gauss codes)
arising from generic closed curves in the plane have the following property:
An even number of other symbols appear between the two occurrences of each
symbol in the word. Thus the looped interlacement graph of a classical knot
diagram must be Eulerian. It follows that every Reidemeister equivalence class
contains infinitely many graphs that are not looped interlacement graphs of
classical knot diagrams. Let $n\geq1$ be an integer, let $E_{n}$ be the
edgeless graph with $n$ vertices, and let $H$ be the connected non-Eulerian
graph obtained from $E_{n}$ by adjoining two vertices with an $\Omega.2$ move.
(The two new vertices should be adjacent if $n$ is even and nonadjacent if $n$
is odd.) Given any graph $G\,$, the disjoint union $G\cup H$ can be obtained
from $G$ with Reidemeister moves: first adjoin $n$ isolated vertices to $G$
using $\Omega.1$ moves, and then use an $\Omega.2$ move to attach the
remaining two vertices of $H$ to the first $n$. As $G\cup H$ is not Eulerian,
it is not the looped interlacement graph of any classical knot diagram.

\medskip

6. \textit{Circle graphs} are the intersection graphs of chord diagrams on
$\mathbb{S}^{1}$; they have received a considerable amount of attention. (See
\cite{Bc} for instance.) If a virtual knot diagram $D$ is obtained from a
classical knot diagram $D^{\prime}$ by designating some crossings as virtual
then $\mathcal{L}(D)$ is the subgraph of $\mathcal{L}(D^{\prime})$ induced by
the vertices that correspond to classical crossings of $D$, so the simple
graph obtained from $\mathcal{L}(D)$ by ignoring all loops is an induced
subgraph of a circle graph. Not all simple graphs are induced subgraphs of
circle graphs -- for instance the wheel graph $W_{5}$ is not, and hence
neither is any graph which contains an induced subgraph isomorphic to $W_{5}$
-- and consequently, many graphs do not arise as looped interlacement graphs
of virtual knot diagrams. We do not know whether or not every graph is
Reidemeister equivalent to the looped interlacement graph of some virtual knot diagram.

\section{Proof of Theorem 2}

To prove Theorem 2 it is necessary to verify that whenever one applies a
Reidemeister move to a (virtual) knot diagram, the effect on the looped
interlacement graph is to apply one of the graph-theoretic Reidemeister moves
described in Definition 7. We need not take the additional moves of virtual
knot theory \cite{Kv} into account, as they do not affect the looped
interlacement graph.%
%TCIMACRO{\FRAME{fhFU}{3.1531in}{3.659in}{0pt}{\Qcb{The first two types of
%knot-theoretic Reidemeister moves involve inserting or removing configurations
%like those in the top and middle rows of the figure. The third type of
%Reidemeister move involves \textquotedblleft moving arc \QTR{it}{AB} through
%crossing \QTR{it}{C},\textquotedblright\ as indicated in the bottom row of the
%figure.}}{\Qlb{Figure 9}}{r_moves.eps}{\special{ language "Scientific Word";
%type "GRAPHIC";  maintain-aspect-ratio TRUE;  display "USEDEF";
%valid_file "F";  width 3.1531in;  height 3.659in;  depth 0pt;
%original-width 3.1073in;  original-height 3.6115in;  cropleft "0";
%filename 'R_moves.eps';file-properties "XNPEU";}} }%
%BeginExpansion
\begin{figure}
[h]
\begin{center}
\includegraphics[
height=3.659in,
width=3.1531in
]%
{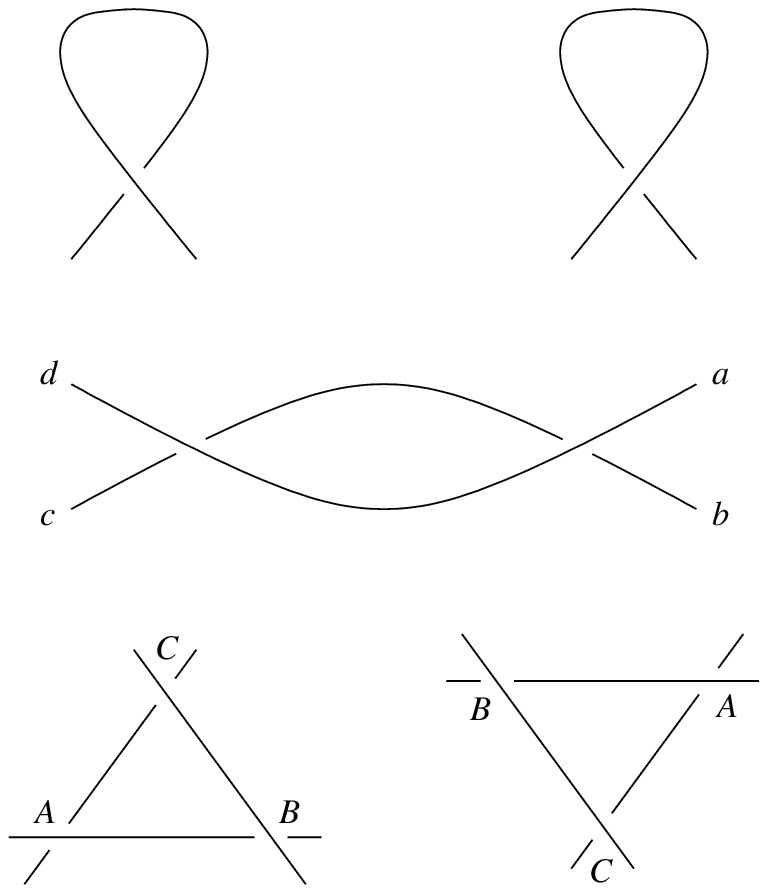}%
\caption{The first two types of knot-theoretic Reidemeister moves involve
inserting or removing configurations like those in the top and middle rows of
the figure. The third type of Reidemeister move involves \textquotedblleft
moving arc \textit{AB} through crossing \textit{C},\textquotedblright\ as
indicated in the bottom row of the figure.}%
\label{Figure 9}%
\end{center}
\end{figure}
%EndExpansion

\medskip

If a new crossing is inserted using a Reidemeister move of type $\Omega.1$,
then the new crossing is not interlaced with any other so the corresponding
vertex of the looped interlacement graph is isolated.

\medskip

Suppose two new crossings are inserted using a Reidemeister move of type
$\Omega.2$ as illustrated in the middle of Figure 9. We trace $D$ by starting
near the point marked $a$ and leaving the pictured portion of $D$. As $D$ is a
knot diagram, we must re-enter the pictured portion first at $b$ or $c$. The
crossings of $D$ that we encounter precisely once in tracing the curve from
$a$ to $b$ or $c$ (whichever we encounter first) are the crossings that are
interlaced with the two pictured ones; consequently the vertices corresponding
to the two pictured crossings have precisely the same neighbors among the
other vertices. As the two pictured crossings are of opposite types, one
corresponds to a looped vertex and the other corresponds to an unlooped
vertex. These two vertices will be adjacent if we first re-enter the pictured
portion of the diagram at $c$, and nonadjacent if we first re-enter the
pictured portion of the diagram at $b$. The situation is essentially the same
if the mirror-image of the Reidemeister move pictured in the middle of Figure
9 is applied.

\medskip

We now discuss Reidemeister moves of type $\Omega.3$. By a
\textit{Reidemeister triangle\/} in a knot diagram we mean a three-sided
complementary region with the following property: One of the three arcs that
form the sides of the region passes above the other two. Call this arc $a$.
Similarly, there is an arc $b$ that passes below the other two, and an arc $c
$ that is \textquotedblleft centered\textquotedblright\ between the other two.
Let $A$ denote the crossing of $b$ and $c$, $B$ the crossing of $a$ and $c$,
and $C$ the crossing of $a$ and $b$. We say such a triangle is
\textit{positive} (\textit{negative}) if $A$, $B$ and $C$ appear in
counter-clockwise (clockwise) order around the triangle. See Figure~10.

\medskip%

%TCIMACRO{\FRAME{fhFU}{3.1548in}{1.1364in}{0pt}{\Qcb{positive and negative
%Reidemeister triangles}}{\Qlb{Figure 10}}{r_triangles.eps}%
%width 3.1548in;  height 1.1364in;  depth 0pt;  original-width 3.1099in;
%original-height 1.1026in;  cropleft "0";  croptop "1";  cropright "1";
%cropbottom "0";  filename 'R_triangles.eps';file-properties "XNPEU";}} }%
%BeginExpansion
\begin{figure}
[h]
\begin{center}
\includegraphics[
height=1.1364in,
width=3.1548in
]%
{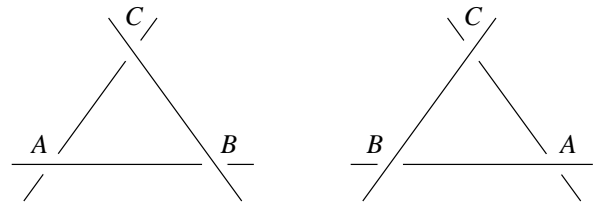}%
\caption{positive and negative Reidemeister triangles}%
\label{Figure 10}%
\end{center}
\end{figure}
%EndExpansion

\medskip

If we orient the knot so that arc $c$ is oriented from $A$ toward $B$, then in
the knot's Gauss code the labels $A$, $B$ and $C$ must appear in one of the
following eight patterns: $ABACBC$, $ABACCB$, $ABBCAC$, $ABBCCA$, $ABCABC $,
$ABCACB$, $ABCBAC$, $ABCBCA$. Furthermore, the signs of the crossings are then
determined, and hence so is the subgraph of the looped interlacement graph
spanned by $A$, $B$ and $C$. Note that the subgraphs arising from a negative
Reidemeister triangle are simply the loop-toggled versions of the subgraphs
arising from a positive triangle. Also note that any crossing that is not part
of the Reidemeister triangle is interlaced with either none or precisely two
of $A$, $B$ and $C$.

\medskip

Every $\Omega.3$ move can be realized by \textquotedblleft moving arc $c$
through crossing $C$\textquotedblright\ in a Reidemeister triangle. See
Figure~9. The effect of such a move is rather minimal. A new Reidemeister
triangle is created, whose arcs and crossings we label using the scheme
described above. The sign of the new triangle is then the same as that of the
original triangle, and the individual crossing signs of $A$, $B$ and $C$ are
unchanged. All that happens is that in each of the eight \textquotedblleft
words\textquotedblright\ above, the first and second letters are transposed,
as are the third and the fourth, and the fifth and the sixth. The
corresponding effect on the looped interlacement graph of the diagram is
simply to toggle all the non-loop edges in the subgraph spanned by $A$, $B$
and $C$.

\medskip

To complete the proof of Theorem~2, it suffices to verify that the six
configurations depicted in Figure~5 summarize all the cases in the discussion
above. The interested reader can verify the following correspondences between
words and graphical configurations. (An apostrophe following a configuration
indicates the toggling of all the non-loop edges in the depiction of that
configuration in Figure~5.) For a positive Reidemeister triangle:
$ABACBC\leftrightarrow(iii)$, $ABACCB\leftrightarrow(ii^{\prime})$,
$ABBCAC\leftrightarrow(iii^{\prime})$, $ABBCCA\leftrightarrow(v^{\prime})$,
$ABCABC\leftrightarrow(v)$, $ABCACB\leftrightarrow(iii)$,
$ABCBAC\leftrightarrow(ii)$, $ABCBCA\leftrightarrow(iii^{\prime})$. For a
negative Reidemeister triangle: $ABACBC\leftrightarrow(iv)$,
$ABACCB\leftrightarrow(i^{\prime})$, $ABBCAC\leftrightarrow(iv^{\prime})$,
$ABBCCA\leftrightarrow(vi^{\prime})$, $ABCABC\leftrightarrow(vi)$,
$ABCACB\leftrightarrow(iv)$, $ABCBAC\leftrightarrow(i)$,
$ABCBCA\leftrightarrow(iv^{\prime})$.

\section{Proof of Theorem 3}

The definition of the graph Jones polynomial can be ``explained''\ in much the
same way Kauffman explained the definition of the knot Jones polynomial in
\cite{K}. Suppose we seek an evaluation of the graph bracket that is invariant
under the second and third types of Reidemeister moves. Consider $G_{1}$ and
$G_{2}$ with $V(G_{1})=\{u,v,w\}=V(G_{2})$, $E(G_{1})=\{\{u,v\},\{v,v\}\}$ and
$E(G_{2})=\{\{u,w\},\{v,w\},\{v,v\}\}$; they have $[G_{1}]=[G_{2}%
]+(d-1)A(B^{2}+ABd+A^{2})$. If an evaluation in an integral domain is to yield
equal values for $[G_{1}]$ and $[G_{2}]$ then one of $d-1$, $A$,
$B^{2}+ABd+A^{2}$ should evaluate to 0; the only option that could possibly
yield an interesting invariant is $B^{2}+ABd+A^{2}\mapsto0$. Now consider
$H_{1}$ and $H_{2}$ with $V(H_{1})=\emptyset=$ $E(H_{1})$, $V(H_{2})=\{v,w\}$
and $E(H_{2})=\{\{v,v\}\}$. They have $[H_{1}]=1$ and $[H_{2}]=AB+d(B^{2}%
+ABd+A^{2})$; if $B^{2}+ABd+A^{2}$ $\mapsto0$ then we are naturally led to the
evaluations $AB\mapsto1$ and (hence) $d\mapsto-A^{2}-A^{-2}$. Finally, the
factor $(-1)^{n}\cdot t^{(3n-6\ell)/4}$ is introduced to cancel variation
under Reidemeister moves of the first type, and to maintain invariance under
moves of the second and third types. Of course this ``explanation''\ only
rationalizes the definition of the graph Jones polynomial; it does not
actually prove that $V_{G}$ is invariant under graph\ Reidemeister moves.

\medskip

The invariance of $V_{G}$ under $\Omega.1$ moves follows directly from
Proposition 1, because the two one-vertex graphs both have graph Jones
polynomial $1$. The recursive description of the graph Jones polynomial that
results from Theorem 1 is complicated by the fact that the appropriate
coefficients $t^{(3n-6\ell)/4}$ for the Jones polynomials of $G^{a}$, $G^{ab}$
and $(G^{ab})^{a}$ will vary according to the configuration of looped vertices
in $G$. Consequently the rest of the proof of Theorem 3 is focused on the
reduced graph bracket rather than the Jones polynomial itself.

\medskip

\begin{proposition}
Suppose $G$ is obtained from the edgeless graph $H=E_{n}$ by adjoining two
vertices $v$ and $w$ in\ an $\Omega.2$ move. Then $\left\langle G\right\rangle
=\left\langle H\right\rangle =(-A^{3})^{n}$ and $V_{G}=V_{H}=1 $.
\end{proposition}

\begin{proof}
If $n=0$ then $V(G)=\{v,w\}$ and $E(G)$ is $\{\{v,v\}\}$ or
$\{\{v,v\},\{v,w\}\}$. Then $[G]$ is $A^{2}d+AB+ABd^{2}+B^{2}d$ or
$A^{2}+ABd+AB+B^{2}$ (respectively), so $\left\langle G\right\rangle $ is
$-A^{4}-1+1+A^{4}+2+A^{-4}-A^{-4}-1=1$ or $A^{2}-A^{-2}-A^{2}+1+A^{-2}=1$; in
either case $V_{G}=(-1)^{2}\cdot t^{(3(2)-6(1))/4}\cdot\left\langle
G\right\rangle =1$.

Proceeding inductively, suppose $n>0$. If $u$ is an isolated vertex of~$G$
then $V_{G}=V_{G-u}$ and the inductive hypothesis implies $V_{G-u}=1$. Hence
we may as well assume that $G$ has no isolated vertex, i.e., that all $n$
vertices of $H$ are adjacent to both $v$ and $w$. Then all the vertices of
$G^{h_{n}w}-h_{n}-w$ and $(G^{h_{n}w})^{h_{n}}-h_{n}-w$ are isolated. As
$G^{h_{n}w}-h_{n}-w$ has one loop and $(G^{h_{n}w})^{h_{n}}-h_{n}-w$ has no
loops, it follows that $\left\langle G^{h_{n}w}-h_{n}-w\right\rangle
=(-1)^{n}A^{3n-6}$ and $\left\langle G^{h_{n}w}-h_{n}-w\right\rangle
=(-1)^{n}A^{3n}$. In addition, $G^{h_{n}}-h_{n}$ differs from $G-h_{n}$ only
in that the loop at $v$ has been moved to $w$ and the adjacency of $v$ and $w
$ has been toggled, so the inductive hypothesis implies that $G^{h_{n}}-h_{n}
$ has Jones polynomial 1 and $\left\langle G^{h_{n}}-h_{n}\right\rangle
=(-A^{3})^{n-1}$.

Part (ii) of the recursive description of the graph bracket tells us that
\begin{align*}
\left\langle G\right\rangle  & =A^{2}\left\langle G^{h_{n}w}-h_{n}%
-w\right\rangle +\left\langle (G^{h_{n}w})^{h_{n}}-h_{n}-w\right\rangle
+A^{-1}\left\langle G^{h_{n}}-h_{n}\right\rangle \\
& =(-1)^{n}A^{3n-4}+(-1)^{n}A^{3n}+A^{-1}(-A^{3})^{n-1}=(-1)^{n}A^{3n}.
\end{align*}

As $G$ has $n+2$ vertices and one loop, it follows that $V_{G}=1$.
\end{proof}

\begin{proposition}
Suppose $G$ is obtained from $H$ by adjoining two vertices $v$ and $w$ in\ an
$\Omega.2$ move. Then $\left\langle G\right\rangle =\left\langle
H\right\rangle $ and $V_{G}=V_{H}$.
\end{proposition}

\begin{proof}
If $H$ has no edges then Proposition 5 applies. The proof proceeds by
induction on $n=\left|  V(H)\right|  $ and for each value of $n\,$, by
induction on the number of loops in $H$.

Suppose $H$ has a loop at a vertex $a$. Then $G-\{a,a\}$ has fewer loops than
$G$ has, and $G^{a}-a$ has fewer vertices. Moreover, $G-\{a,a\}$ and $G^{a}-a$
are obtained from $H-\{a,a\}$ and $H^{a}-a$ (respectively) by Reidemeister
moves of type $\Omega.2$, so we may inductively assume that $\left\langle
G-\{a,a\}\right\rangle =\left\langle H-\{a,a\}\right\rangle $ and
$\left\langle G^{a}-a\right\rangle =\left\langle H^{a}-a\right\rangle $. Then
$\left\langle G\right\rangle =A^{-2}\left\langle G-\{a,a\}\right\rangle
+(A-A^{-3})\left\langle G^{a}-a\right\rangle =A^{-2}\left\langle
H-\{a,a\}\right\rangle +(A-A^{-3})\left\langle H^{a}-a\right\rangle
=\left\langle H\right\rangle $; as $G$ has two more vertices and one more loop
than $H$, this implies $V_{G}=V_{H}$.

Suppose now that $H$ has no loops and $a,b\in V(H)$ are adjacent. Then
$G^{ab}-a-b,(G^{ab})^{a}-a-b$ and $G^{a}-a$ are obtained from $H^{ab}%
-a-b,(H^{ab})^{a}-a-b$ and $H^{a}-a$ (respectively) by Reidemeister moves of
type $\Omega.2$, so the inductive hypothesis tells us that $\left\langle
G^{ab}-a-b\right\rangle =\left\langle H^{ab}-a-b\right\rangle $, $\left\langle
(G^{ab})^{a}-a-b\right\rangle =\left\langle (H^{ab})^{a}-a-b\right\rangle $
and $\left\langle G^{a}-a\right\rangle =\left\langle H^{a}-a\right\rangle $.
Then $\left\langle G\right\rangle =A^{2}\left\langle G^{ab}-a-b\right\rangle
+\left\langle (G^{ab})^{a}-a-b\right\rangle +A^{-1}\left\langle G^{a}%
-a\right\rangle =A^{2}\left\langle H^{ab}-a-b\right\rangle +\left\langle
(H^{ab})^{a}-a-b\right\rangle +A^{-1}\left\langle H^{a}-a\right\rangle
=\left\langle H\right\rangle $. As $G$ has one loop and $H$ has none, this
implies $V_{G}=V_{H}$.
\end{proof}

\begin{lemma}
Suppose $m\geq1$. Let $\Gamma$ have $V(\Gamma)=\{v_{1},...,v_{m+1}\}$ and
$E(\Gamma)=\{v_{1},v_{2}\}\cup\{\{v_{i},v_{1}\},\{v_{i},v_{2}\}$
%TCIMACRO{\TEXTsymbol{\vert} }%
%BeginExpansion
$\vert$
%EndExpansion
$3\leq i\leq m+1\}$, and let $\Gamma^{\prime}$ have $V(\Gamma^{\prime
})=\{v_{1},...,v_{m+2}\}$ and $E(\Gamma^{\prime})=\{\{v_{i},v_{1}%
\},\{v_{i},v_{2}\}$
%TCIMACRO{\TEXTsymbol{\vert} }%
%BeginExpansion
$\vert$
%EndExpansion
$3\leq i\leq m+2\}$. Then $V_{\Gamma}=V_{\Gamma^{\prime}}$.
\end{lemma}

\begin{proof}
$(\Gamma^{\prime})^{v_{m+2}v_{1}}$ is obtained from $\Gamma^{\prime}$ by
toggling away all the adjacencies between $v_{2}$ and the $v_{i}$, $2<i<m+2$;
hence $(\Gamma^{\prime})^{v_{m+2}v_{1}}-v_{1}-v_{m+2}=E_{m}$ and
$((\Gamma^{\prime})^{v_{m+2}v_{1}})^{v_{m+2}}-v_{1}-v_{m+2}=E_{m-1}\cup L_{1}%
$, where $L_{1}$ consists of a single looped vertex. $(\Gamma^{\prime
})^{v_{m+2}}-v_{m+2}$ is obtained from $\Gamma^{\prime}-v_{m+2}$ by adjoining
an edge $\{v_{1},v_{2}\}$ and also adjoining loops at $v_{1}$ and $v_{2}$.
Hence $((\Gamma^{\prime})^{v_{m+2}}-v_{m+2})-\{v_{1},v_{1}\}$ is obtained from
$E_{m-1}$ by adjoining $v_{1}$ and $v_{2}$ in an $\Omega.2$ move, and
$((\Gamma^{\prime})^{v_{m+2}}-v_{m+2})^{v_{1}}-v_{1}$ is the graph in which
$v_{2}$ is unlooped and isolated and $v_{3},...,v_{m+1}$ are all looped and
all adjacent to each other; that is, $((\Gamma^{\prime})^{v_{m+2}}%
-v_{m+2})^{v_{1}}-v_{1}$ is the disjoint union of an isomorphic copy of
$E_{1}$ and an isomorphic copy of $K_{m-1}+I$, the complete looped graph with
$m-1$ vertices.

We conclude that
\begin{align*}
\left\langle \Gamma^{\prime}\right\rangle  & =A^{2}\left\langle (\Gamma
^{\prime})^{v_{m+2}v_{1}}-v_{1}-v_{m+2}\right\rangle +\left\langle
((\Gamma^{\prime})^{v_{m+2}v_{1}})^{v_{m+2}}-v_{1}-v_{m+2}\right\rangle \\
& +A^{-1}\left\langle (\Gamma^{\prime})^{v_{m+2}}-v_{m+2}\right\rangle \\
& =A^{2}(-A^{3})^{m}+(-A^{3})^{m-1}(-A^{-3})+A^{-3}\left\langle ((\Gamma
^{\prime})^{v_{m+2}}-v_{m+2})-\{v_{1},v_{1}\}\right\rangle \\
& +A^{-1}(A-A^{-3})\left\langle ((\Gamma^{\prime})^{v_{m+2}}-v_{m+2})^{v_{1}%
}-v_{1}\right\rangle \\
& =A^{2}(-A^{3})^{m}+(-A^{3})^{m-1}(-A^{-3})+A^{-3}(-A^{3})^{m-1}\\
& +A^{-1}(A-A^{-3})(-A^{3})\left\langle K_{m-1}+I\right\rangle \\
& =(-1)^{m}A^{3m+2}+A^{-1}(A-A^{-3})(-A^{3})\left\langle K_{m-1}%
+I\right\rangle .
\end{align*}

$\Gamma^{v_{m+1}v_{1}}$ is obtained from $\Gamma$ by toggling away all the
adjacencies between $v_{2}$ and the $v_{i}$, $3\leq i\leq m$; hence
$\Gamma^{v_{m+1}v_{1}}-v_{m+1}-v_{1}=E_{m-1}$ and $(\Gamma^{v_{m+1}v_{1}%
})^{v_{m+1}}-v_{m+1}-v_{1}=E_{m-2}\cup L_{1}$. $\Gamma^{v_{m+1}}-v_{m+1}$ is
obtained from $\Gamma-v_{m+1}$ by removing the edge $\{v_{1},v_{2}\}$ and
adjoining loops at $v_{1}$ and $v_{2}$. Consequently $(\Gamma^{v_{m+1}%
}-v_{m+1})-\{v_{1},v_{1}\}$ is obtained from $E_{m-2}$ by adjoining $v_{1}$
and $v_{2}$ in an $\Omega.2$ move, and $(\Gamma^{v_{m+1}}-v_{m+1})^{v_{1}%
}-v_{1}$ is the graph in which $v_{2},...,v_{m}$ are all looped and all
adjacent to each other; that is, $(\Gamma^{v_{m+1}}-v_{m+1})^{v_{1}}-v_{1}$ is
an isomorphic copy of $K_{m-1}+I$.

We conclude that
\begin{align*}
\left\langle \Gamma\right\rangle  & =A^{2}\left\langle \Gamma^{v_{m+1}v_{1}%
}-v_{1}-v_{m+1}\right\rangle +\left\langle (\Gamma^{v_{m+1}v_{1}})^{v_{m+1}%
}-v_{m+1}-v_{1}\right\rangle \\
& +A^{-1}\left\langle \Gamma^{v_{m+1}}-v_{m+1}\right\rangle \\
& =A^{2}(-A^{3})^{m-1}+(-A^{3})^{m-2}(-A^{-3})+A^{-3}\left\langle
((\Gamma^{v_{m+1}}-v_{m+1})-\{v_{1},v_{1}\}\right\rangle \\
& +A^{-1}(A-A^{-3})\left\langle (\Gamma^{v_{m+1}}-v_{m+1})^{v_{1}}%
-v_{1}\right\rangle \\
& =A^{2}(-A^{3})^{m-1}+(-A^{3})^{m-2}(-A^{-3})+A^{-3}(-A^{3})^{m-2}\\
& +A^{-1}(A-A^{-3})\left\langle K_{m-1}+I\right\rangle \\
& =(-1)^{m-1}A^{3m-1}+A^{-1}(A-A^{-3})\left\langle K_{m-1}+I\right\rangle
\text{,}%
\end{align*}

\noindent and hence $\left\langle \Gamma^{\prime}\right\rangle =-A^{3}%
\cdot\left\langle \Gamma\right\rangle $. As $\Gamma^{\prime}$ has one more
vertex than $\Gamma$ and neither has any loops, this implies $V_{\Gamma
^{\prime}}=V_{\Gamma}$.
\end{proof}

\begin{proposition}
Suppose $G$ has three vertices $u,v,w$ which span a subgraph $H$ isomorphic to
one of those pictured in Figure 5; suppose further that every vertex outside
$H$ is unlooped, has degree 2 and is adjacent to two of $u,v,w$. Let
$G^{\prime}$ be the graph obtained from $G$ by toggling all the non-loop
adjacencies in $H$. Then $\left\langle G\right\rangle =\left\langle G^{\prime
}\right\rangle $ and hence $V_{G}=V_{G^{\prime}}$.
\end{proposition}

\begin{proof}
In case $G=H$ the result can be verified by direct computation, which we leave
to the reader. (It is interesting to note that even in this simple case, the
proposition fails for the 3-vertex configurations not pictured in Figure 5.)

We proceed inductively, assuming that $G$ has $n\geq4$ vertices. Let
$V(G)-V(H)=N_{uv}\cup N_{uw}\cup N_{vw}$, with the elements of $N_{ij}$
adjacent to $i$ and $j$.

Case 1. Suppose $H$ is isomorphic to the graph appearing in Figure 5
\textit{vi}, with a loop at $u$.

Suppose $a\in N_{uv}$. Then $G^{av}$ is obtained from $G$ by toggling all
adjacencies between $u$ and neighbors of $v$ other than $a$ and $u$; hence
$V(G^{av}-a-v)=\{u,w\}\cup(N_{uv}-\{a\})\cup N_{uw}\cup N_{vw}$ and
$E(G^{av}-a-v)=E(G-a-v)\cup\{\{u,y\}|y\in N_{vw}\}-\{\{u,y\}|y\in
N_{uv}\}-\{u,w\}$. Similarly, $(G^{\prime})^{av}$ is obtained from $G^{\prime
}$ by toggling all adjacencies between $u$ and neighbors of $v$ other than
$a$; hence $V((G^{\prime})^{av}-a-v)=\{u,w\}\cup(N_{uv}-\{a\})\cup N_{uw}\cup
N_{vw}$ and $E((G^{\prime})^{av}-a-v)=E(G^{\prime}-a-v)\cup\{\{u,y\}|y\in
N_{vw}\}-\{\{u,y\}|y\in N_{uv}\}$. We see that $G^{av}-a-v=(G^{\prime}%
)^{av}-a-v$. The only differences between $G^{av}$ and $(G^{av})^{a}$ are that
the latter has a loop at $v$ instead of $u$ and no edge $\{u,v\}$;
consequently $(G^{av})^{a}-a-v$ coincides with $G^{av}-a-v$ except for the
fact that $(G^{av})^{a}-a-v$ has no loop at $u$. Similarly, the only
difference between $(G^{\prime})^{av}-a-v$ and $((G^{\prime})^{av})^{a}-a-v$
is that the latter has no loop at $u$; hence $(G^{av})^{a}-a-v=((G^{\prime
})^{av})^{a}-a-v$.

$G^{a}$ is obtained from $G$ by moving the loop from $u$ to $v$ and removing
the edge $uv$, and $(G^{\prime})^{a}$ is obtained from $G^{\prime}$ by moving
the loop from $u$ to $v$ and adjoining an edge $\{u,v\}$. Consequently the
full subgraph of $G^{a}-a$ with vertices $u,v,w$ is isomorphic to the graph
pictured in Figure 5 \textit{iv}, and $(G^{\prime})^{a}-a$ is obtained from
$G^{a}-a$ by performing a Reidemeister move of type $\Omega.3$ on this
subgraph. The inductive hypothesis tells us that $V_{G^{a}-a}=V_{(G^{\prime
})^{a}-a}$; as $G^{a}-a$ and $(G^{\prime})^{a}-a$ both have one loop and $n-1$
vertices, this implies that $\left\langle G^{a}-a\right\rangle =\left\langle
(G^{\prime})^{a}-a\right\rangle $. We conclude that
\begin{align*}
\left\langle G\right\rangle  & =A^{2}\left\langle G^{av}-a-v\right\rangle
+\left\langle (G^{av})^{a}-a-v\right\rangle +A^{-1}\left\langle G^{a}%
-a\right\rangle \\
& =A^{2}\left\langle (G^{\prime})^{av}-a-v\right\rangle +\left\langle
((G^{\prime})^{av})^{a}-a-v\right\rangle +A^{-1}\left\langle (G^{\prime}%
)^{a}-a\right\rangle =\left\langle G^{\prime}\right\rangle .
\end{align*}

If there is an $a\in N_{uw}$ the same argument applies, with $v$ and $w$
interchanged throughout.

Suppose $N_{uv}=\emptyset=N_{uw}$, so that $V(G)=\{u,v,w\}\cup N_{vw}$; let
$a\in N_{vw}$ and let $m=n-3$. Then $G$ is obtained from the graph denoted
$\Gamma$ in Lemma 1 by adjoining $u$ and $a$ through an $\Omega.2$ move, and
$G^{\prime}$ is obtained from the graph denoted $\Gamma^{\prime}$ in Lemma 1
by adjoining the isolated, looped vertex $u$. It follows that $V_{G}%
=V_{\Gamma}=V_{\Gamma^{\prime}}=V_{G^{\prime}}$.

Case 2. Suppose $H$ is isomorphic to the graph appearing in Figure 5
\textit{ii}, with $v$ adjacent to both $u$ and $w$.

Suppose $a\in N_{uv}$. Then $G^{av}$ is obtained from $G$ by toggling all
adjacencies between $u$ and neighbors of $v$ other than $a$ and $u$; hence
$V(G^{av}-a-v)=\{u,w\}\cup(N_{uv}-\{a\})\cup N_{uw}\cup N_{vw}$ and
$E(G^{av}-a-v)=E(G-a-v)\cup\{\{u,w\}\}\cup\{\{u,y\}|y\in N_{vw}%
\}-\{\{u,y\}|y\in N_{uv}\}$. Similarly, $(G^{\prime})^{av}$ is obtained from
$G^{\prime}$ by toggling all adjacencies between $u$ and neighbors of $v$
other than $a$; hence $V((G^{\prime})^{av}-a-v)=\{u,w\}\cup(N_{uv}-\{a\})\cup
N_{uw}\cup N_{vw}$ and $E((G^{\prime})^{av}-a-v)=E(G^{\prime}-a-v)\cup
\{\{u,y\}|y\in N_{vw}\}-\{\{u,y\}|y\in N_{uv}\}$. We see that $G^{av}%
-a-v=(G^{\prime})^{av}-a-v$. The only differences between $G^{av}$ and
$(G^{av})^{a}$ are that the latter has loops at $u$ and $v$ and no edge
$\{u,v\}$; consequently $(G^{av})^{a}-a-v$ coincides with $G^{av}-a-v$ except
for the fact that $(G^{av})^{a}-a-v$ has a loop at $u$. Similarly, the only
difference between $(G^{\prime})^{av}-a-v$ and $((G^{\prime})^{av})^{a}-a-v$
is that the latter has a loop at $u$; hence $(G^{av})^{a}-a-v=((G^{\prime
})^{av})^{a}-a-v$.

$G^{a}$ is obtained from $G$ by adjoining loops at $u$ and $v$ and removing
the edge $\{u,v\}$, and $(G^{\prime})^{a}$ is obtained from $G^{\prime}$ by
adjoining loops at $u$ and $v$ and also adjoining an edge $\{u,v\}$.
Consequently the full subgraph of $(G^{\prime})^{a}-a$ with vertices $u,v,w$
is isomorphic to the graph pictured in Figure 5 \textit{iii}, and $G^{a}-a$ is
obtained from $(G^{\prime})^{a}-a$ by performing an $\Omega.3$ move on this
subgraph. The inductive hypothesis tells us that $V_{G^{a}-a}=V_{(G^{\prime
})^{a}-a}$; as $G^{a}-a$ and $(G^{\prime})^{a}-a$ both have two loops and
$n-1$ vertices, this implies that $\left\langle G^{a}-a\right\rangle
=\left\langle (G^{\prime})^{a}-a\right\rangle $. We conclude that
\begin{align*}
\left\langle G\right\rangle  & =A^{2}\left\langle G^{av}-a-v\right\rangle
+\left\langle (G^{av})^{a}-a-v\right\rangle +A^{-1}\left\langle G^{a}%
-a\right\rangle \\
& =A^{2}\left\langle (G^{\prime})^{av}-a-v\right\rangle +\left\langle
((G^{\prime})^{av})^{a}-a-v\right\rangle +A^{-1}\left\langle (G^{\prime}%
)^{a}-a\right\rangle =\left\langle G^{\prime}\right\rangle .
\end{align*}

If there is an $a\in N_{vw}$ the same argument applies, with $u$ and $w$
interchanged throughout.

Suppose $N_{uv}=\emptyset=N_{vw}$, so that $V(G)=\{u,v,w\}\cup N_{uw}$; let
$m=n-2$. Then $G$ is the graph denoted $\Gamma^{\prime}$ in Lemma 1, and
$G^{\prime}$ is obtained from the graph denoted $\Gamma$ in Lemma 1 by
adjoining the isolated, unlooped vertex $v$. It follows that $V_{G}%
=V_{\Gamma^{\prime}}=V_{\Gamma}=V_{G^{\prime}}$.

Case 3. Suppose $H$ is isomorphic to the graph appearing in Figure 5
\textit{iv}, with a loop at $u$ and $v$ not adjacent to $u$. Then $G-\{u,u\}$
and $G^{\prime}-\{u,u\}$ are related as in Case 2, so $\left\langle
G-\{u,u\}\right\rangle =\left\langle G^{\prime}-\{u,u\}\right\rangle $.

Observe that $G^{u}-u$ and $(G^{\prime})^{u}-u$ share the same subgraph $S$
spanned by $N_{uv}\cup N_{uw}\cup N_{vw}$: loops and non-loop adjacencies
within $N_{uv}\cup N_{uw}$ are toggled from those of $G$, and loops and
adjacencies involving elements of $N_{vw}$ are the same as those of $G$. In
$G^{u}-u$, all elements of $N_{uv}\cup N_{vw}$ are adjacent to both $v$ and $w
$, and no element of $N_{uw}$ is adjacent to either $v$ or $w$. $G^{u}-u$ has
a loop at $w$ and no loop at $v$, so $G^{u}-u$ is obtained from $S$ by
adjoining $v$ and $w$ in an $\Omega.2$ move. In $(G^{\prime})^{u}-u$, all
elements of $N_{uw}\cup N_{vw}$ are adjacent to both $v$ and $w$, and no
element of $N_{uv}$ is adjacent to either $v$ or $w$. $(G^{\prime})^{u}-u$ has
a loop at $v$ and no loop at $w$, so $(G^{\prime})^{u}-u$ is obtained from $S$
by adjoining $v$ and $w$ in an $\Omega.2$ move. It follows that $\left\langle
G^{u}-u\right\rangle =\left\langle (G^{\prime})^{u}-u\right\rangle $.

We conclude that
\begin{align*}
\left\langle G\right\rangle  & =A^{-2}\left\langle G-\{u,u\}\right\rangle
+(A-A^{-3})\left\langle G^{u}-u\right\rangle \\
& =A^{-2}\left\langle G^{\prime}-\{u,u\}\right\rangle +(A-A^{-3})\left\langle
(G^{\prime})^{u}-u\right\rangle =\left\langle G^{\prime}\right\rangle .
\end{align*}

As the graphs appearing on the left-hand side of Figure 5 result from toggling
the loops in the graphs that appear on the right-hand side, the remaining
three cases follows from the first three and Proposition 2.
\end{proof}

\begin{proposition}
Suppose $G$ has three vertices $u,v,w$ which span a subgraph $H$ isomorphic to
one of those pictured in Figure 5; suppose further that every vertex outside
$H$ that is adjacent to any of $u,v,w$ is adjacent to precisely two of
$u,v,w$. If $G^{\prime}$ is the graph obtained from $G$ by toggling all the
non-loop adjacencies in $H$ then $\left\langle G\right\rangle =\left\langle
G^{\prime}\right\rangle $ and $V_{G}=V_{G^{\prime}}$.
\end{proposition}

\begin{proof}
If $G$ has no vertices outside $H$ then we appeal to Proposition 7; the proof
proceeds by induction on $\left\vert V(G)\right\vert =n\geq4$.

If $\left|  V(G)\right|  =n$ and every edge of $G$ is incident on $H$ then we
appeal to Proposition 7 (and Proposition 1, if necessary); we proceed by
induction on the number of edges of $G$ not incident on $H$.

Suppose $G$ has a looped vertex $a\not \in V(H)$. The inductive hypothesis
tells us that $\left\langle G-\{a,a\}\right\rangle =\left\langle G^{\prime
}-\{a,a\}\right\rangle $. If $a$ is not adjacent to any of $u,v,w$ then local
complementation at $a$ does not affect any edges incident on $H$, so $G^{a}-a$
and $(G^{\prime})^{a}-a$ are related through a type 3 Reidemeister move
represented by the same part of Figure 5 as $G$ and $G^{\prime}$. Suppose $a$
is adjacent to precisely two of $u,v,w$. If $b\in V(G)-\{u,v,w,a\}$ is not
adjacent to $a$, then local complementation at $a$ does not affect adjacencies
between $b$ and $u,v,w$. If $b\in V(G)-\{u,v,w,a\}$ is adjacent to $a$, then
local complementation at $a$ toggles two of the adjacencies between $b$ and
$u,v,w$. In any case, the fact that $b$ is adjacent to an even number of
$u,v,w$ in $G$ implies that $b $ is also adjacent to an even number of $u,v,w$
in $G^{a}$. Local complementation at $a$ transforms $H$ into another one of
the three-vertex configurations of Figure 5; for instance, if $H$ is
isomorphic to Figure 5 \textit{iii} and $a$ is adjacent to the two looped
vertices then the subgraph of $(G^{\prime})^{a}-a$ spanned by $u,v,w$ is
isomorphic to Figure 5 \textit{ii}. We conclude that $G^{a}-a$ and
$(G^{\prime})^{a}-a$ are related through an $\Omega.3$ move; as both have only
$\left\vert V(G)\right\vert -1$ vertices, we may assume inductively that
$\left\langle G^{a}-a\right\rangle =\left\langle (G^{\prime})^{a}%
-a\right\rangle $ and hence
\begin{align*}
\left\langle G\right\rangle  & =A^{-2}\left\langle G-\{a,a\}\right\rangle
+(A-A^{-3})\left\langle G^{a}-a\right\rangle \\
& =A^{-2}\left\langle G^{\prime}-\{a,a\}\right\rangle +(A-A^{-3})\left\langle
(G^{\prime})^{a}-a\right\rangle =\left\langle G^{\prime}\right\rangle .
\end{align*}

Suppose $G$ has no loops but has an edge $\{a,b\}$ not incident on $H$. As in
the preceding paragraph, $G^{a}-a$ and $(G^{\prime})^{a}-a$ are also related
through a Reidemeister move of type 3 so we may assume inductively that
$\left\langle G^{a}-a\right\rangle =\left\langle (G^{\prime})^{a}%
-a\right\rangle $. If either $a$ or $b$ is adjacent to none of $u,v,w$ then
passing from $G$ to $G^{ab}$ does not affect $H$; the same is true if $a $ and
$b$ are adjacent to the same two of $u,v,w$. Suppose $a$ and $b$ are adjacent
to different pairs of vertices of $H$. Then one of $u,v,w$ is adjacent to $a$
and $b$, one is adjacent to $a$ and not $b$, and one is adjacent to $b$ and
not $a$. If $x\in V(G)-\{a,b,u,v,w\}$ is adjacent to neither $a$ nor $b$ then
passing from $G$ to $G^{ab}$ does not affect any adjacencies between $x$ and
$u,v,w$; if $x$ is adjacent to $a$ or $b$ then passing from $G$ to $G^{ab}$
toggles precisely two adjacencies between $x$ and $u,v,w$ (the two whose
adjacencies with $a,b$ do not match those of $x$). Consequently the fact that
$x$ is adjacent to an even number of $u,v,w$ in $G$ implies that $x$ is also
adjacent to an even number of $u,v,w$ in $G^{ab}$. As all three non-loop
adjacencies involving $u,v,w$ are toggled in passing from $G$ to $G^{ab}$, we
conclude that $(G^{\prime})^{ab}$ and $G^{ab}$ are related through the same
$\Omega.3$ Reidemeister move as $G$ and $G^{\prime}$, respectively. As in the
preceding paragraph, it follows that $(G^{ab})^{a}-a$ and $((G^{\prime}%
)^{ab})^{a}-a$ are also related through $\Omega.3$ moves. We conclude
inductively that
\begin{align*}
\left\langle G\right\rangle  & =A^{2}\left\langle G^{ab}-a-b\right\rangle
+\left\langle (G^{ab})^{a}-a-b\right\rangle +A^{-1}\left\langle G^{a}%
-a\right\rangle \\
& =A^{2}\left\langle (G^{\prime})^{ab}-a-b\right\rangle +\left\langle
((G^{\prime})^{ab})^{a}-a-b\right\rangle +A^{-1}\left\langle (G^{\prime}%
)^{a}-a\right\rangle =\left\langle G^{\prime}\right\rangle .
\end{align*}

\end{proof}

\begin{center}
\textbf{Acknowledgments}
\end{center}

\medskip

We thank the anonymous referee for several comments and corrections that
significantly improved the paper. We are also grateful to Lafayette College
for supporting our work.

\end{document}